\documentclass[a4paper]{amsart}
\usepackage{verbatim}
\usepackage[dvips]{color}
\usepackage{amssymb}
\usepackage[active]{srcltx}
\usepackage[pdfpagelabels,pdfstartview = FitH,bookmarksopen = true,bookmarksnumbered = true,plainpages = false,hypertexnames = false] {hyperref}

\author[T. Kuusi]{Tuomo Kuusi}
\address{Tuomo Kuusi\\Aalto University
Institute of Mathematics
\\ P.O. Box 111000
FI-00076 Aalto,
Finland}
\email{tuomo.kuusi@tkk.fi}

\author[G. Mingione]{Giuseppe Mingione}
\address{Dipartimento di Matematica e Informatica, Universit\`a di Parma\\
Parco Area delle Scienze 53/a, Campus, 43124 Parma, Italy}
\email{giuseppe.mingione@unipr.it.}

\allowdisplaybreaks

\DeclareMathOperator*{\osc}{osc}

\def\osc{\operatorname{osc}}

\newtheorem{theorem}{Theorem}[section]
\newtheorem{prop}{Proposition}[section]
\newtheorem{lemma}{Lemma}[section]
\newtheorem{cor}{Corollary}[section]
\theoremstyle{definition}
\newtheorem{definition}{Definition}
\newtheorem{remark}{Remark}[section]

\numberwithin{equation}{section}

\newcommand{\rif}[1]{(\ref{#1})}

\def\eqn#1$$#2$${\begin{equation}\label#1#2\end{equation}}
\def\charfn_#1{{\raise1.2pt\hbox{$\chi
_{\kern-1pt\lower3pt\hbox{{$\scriptstyle#1$}}}$}}}

\def\qq1{q_*}
\def\q2{q_{**}}

\def\ep{\varepsilon}
\def\en{\mathbb N}
\def\er{\mathbb R}

\def\loc{\operatorname{loc}}
\def\osc{\operatorname{osc}}

\newdimen\vintbar
\def\vint{-\kern-\vintbar\int}

\def\B{\mathcal B}

\def\0{\boldsymbol 0}

\newcommand{\ratio}{\nu, L}

\newcommand{\divo}{\textnormal{div}}

 \newcommand{\mean}[1]{-\hskip-1.08em\int_{#1}}

\newcommand{\trif}[1] {\textnormal{\rif{#1}}}

\newtoks\by
\newtoks\paper
\newtoks\book
\newtoks\jour
\newtoks\yr
\newtoks\pages
\newtoks\vol
\newtoks\publ
\def\et{ \& }
\def\name[#1, #2]{#1 #2}
\def\ota{{\hbox{\bf ???}}}
\def\cLear{\by=\ota\paper=\ota\book=\ota\jour=\ota\yr=\ota
\pages=\ota\vol=\ota\publ=\ota}
\def\endpaper{\the\by, \textit{\the\paper},
{\the\jour} \textbf{\the\vol} (\the\yr), \the\pages.\cLear}
\def\endbook{\the\by, \textit{\the\book},
\the\publ, \the\yr.\cLear}
\def\endpap{\the\by, \textit{\the\paper}, \the\jour.\cLear}
\def\endproc{\the\by, \textit{\the\paper}, \the\book, \the\publ,
\the\yr, \the\pages.\cLear}

\def\mean#1{\mathchoice%
          {\mathop{\kern 0.2em\vrule width 0.6em height 0.69678ex depth -0.58065ex
                  \kern -0.8em \intop}\nolimits_{\kern -0.4em#1}}%
          {\mathop{\kern 0.1em\vrule width 0.5em height 0.69678ex depth -0.60387ex
                  \kern -0.6em \intop}\nolimits_{#1}}%
          {\mathop{\kern 0.1em\vrule width 0.5em height 0.69678ex
              depth -0.60387ex
                  \kern -0.6em \intop}\nolimits_{#1}}%
          {\mathop{\kern 0.1em\vrule width 0.5em height 0.69678ex depth -0.60387ex
                  \kern -0.6em \intop}\nolimits_{#1}}}

\def\vintslides_#1{\mathchoice%
          {\mathop{\kern 0.1em\vrule width 0.5em height 0.697ex depth -0.581ex
                  \kern -0.6em \intop}\nolimits_{\kern -0.4em#1}}%
          {\mathop{\kern 0.1em\vrule width 0.3em height 0.697ex depth -0.604ex
                  \kern -0.4em \intop}\nolimits_{#1}}%
          {\mathop{\kern 0.1em\vrule width 0.3em height 0.697ex depth -0.604ex
                  \kern -0.4em \intop}\nolimits_{#1}}%
          {\mathop{\kern 0.1em\vrule width 0.3em height 0.697ex depth -0.604ex
                  \kern -0.4em \intop}\nolimits_{#1}}}

\newcommand{\aveint}[2]{\mathchoice%
          {\mathop{\kern 0.2em\vrule width 0.6em height 0.69678ex depth -0.58065ex
                  \kern -0.8em \intop}\nolimits_{\kern -0.45em#1}^{#2}}%
          {\mathop{\kern 0.1em\vrule width 0.5em height 0.69678ex depth -0.60387ex
                  \kern -0.6em \intop}\nolimits_{#1}^{#2}}%
          {\mathop{\kern 0.1em\vrule width 0.5em height 0.69678ex depth -0.60387ex
                  \kern -0.6em \intop}\nolimits_{#1}^{#2}}%
          {\mathop{\kern 0.1em\vrule width 0.5em height 0.69678ex depth -0.60387ex
                  \kern -0.6em \intop}\nolimits_{#1}^{#2}}}

\newcommand{\eps}{\varepsilon}

\renewcommand{\osc}{\operatornamewithlimits{osc}}

\title[Riesz potentials and nonlinear parabolic equations]{Riesz potentials and \\ nonlinear parabolic equations}
\begin{document}

\maketitle

\centerline{To Neil Trudinger for his 70th birthday}

\begin{abstract} The spatial gradient of solutions to nonlinear degenerate parabolic equations can be pointwise estimated by the caloric Riesz potential of the right hand side datum, exactly as in the case of the heat equation. Heat kernels type estimates persist in the nonlinear case. \end{abstract}

\tableofcontents

\section{Results}
In this paper we are going to consider nonlinear, possibly degenerate parabolic equations whose model is given by
\eqn{modello}
$$
u_t- \divo\, (|Du|^{p-2}Du)=\mu%%\,, \qquad p \geq 2\,, %%%TUOMO : added p \geq 2
$$
where $\mu$ denotes in the most general case a Borel measure with finite total mass. Although the kind of problems considered are nonlinear our goal is to provide a suitable {\em linear potential theory} aimed at describing, in a sharp way, the regularity properties of the gradient $Du$ in terms of those of $\mu$. More precisely, our description shows that sharp gradient pointwise estimates can be given in terms of classical Riesz caloric potentials of the right hand side $\mu$. We will see that, surprisingly enough,
%%%TUOMO : same -> similar %%%ROS rechange here
bounds similar to those that hold for the heat equation
\eqn{heateq}
$$
u_t - \triangle u =\mu
$$
actually hold for solutions to \rif{modello} and, more in general, for solutions to quasilinear equations of the type
\eqn{maineq}
$$
u_t- \divo\, a(x,t,Du)=\mu \qquad \mbox{in} \ \ \Omega_T = \Omega \times (-T,0)\,,
$$
provided suitable, actually optimal, regularity assumptions are made on the partial map $(x,t) \mapsto a(t, x,z)$. Here $\Omega\subset \er^n$ is an open subset, $n \geq 2$ and $T >0$.  Specifically, we shall consider a Carath\'eodory vector field $a \colon \Omega_T \times \er^n \to \er^n$ which is $C^{1}$-regular in the third variable and
satisfying the following parabolicity and continuity conditions: \eqn{asp}
$$
\left\{
    \begin{array}{c}
    |a(x,t,z)|+|\partial_z a(x,t,z)|(|z|^2+s^2)^{1/2} \leq L(|z|^2+s^2)^{(p-1)/2} \\ [5 pt]
    \nu(|z|^2+s^2)^{(p-2)/2}|\xi|^{2} \leq \langle \partial_z a(x,t,z)\xi, \xi
    \rangle \\[5 pt]
    |a(x,t,z)-a(x_0,t,z)|\leq L \omega(|x-x_0|)(|z|^2+s^2)^{(p-1)/2}
    \end{array}
    \right.
$$
whenever $z,\xi \in \er^n$, $x, x_0 \in \Omega$, $t\in (-T,0)$, where $0 < \nu \leq L $ are positive numbers and $p\geq 2$. The symbol $\omega(\cdot)$ denotes a modulus of continuity meaning that $\omega\colon [0,\infty) \to [0,1]$ is a nondecreasing concave function such that $\omega(0)=0$. In the following $s\geq 0$ is a parameter that will be used to distinguish the degenerate case ($s=0$), which covers the model equation in \rif{modello}, from the non-degenerate one ($s>0$); the analysis made in the following will see no difference between these two cases. In the rest of the paper we shall assume that the partial map $$x\to \frac{a(x,t,z)}{(|z|^2+s^2)^{(p-1)/2}}$$  is Dini-continuos in the sense that
\eqn{aspd}
$$
\int_0^1 \omega(\varrho) \, \frac{d \varrho}{\varrho}< \infty\,.
$$
This assumption is optimal for the estimates we are going to derive in the following; %%%TUOMO : ; %%%ROS rechange in the following
see the comments at the beginning of Section \ref{compi}.  
We anyway remark that, everywhere in this paper, we only assume measurability of the partial map $t\to a(x,t,z)$, in other words we assume that time coefficients are merely measurable.
Yet, in this paper we shall always consider the case
$$
p \geq 2
$$
as the case $p < 2$ has already been treated elsewhere \cite{KMpisa} and involves an analysis which is different and somewhat simpler than the one which is necessary here. For more notation we refer the reader to Section 2 and to the rest of this introductory section.
%%%ROS rewording line above
\begin{remark}[On the notion of solution]\label{addiass}
%%%TUOMO : rewording
Throughout the paper, when considering weak solutions to \rif{maineq} and unless otherwise stated, we shall mean energy weak solutions. An energy weak solution $u$ belongs to the parabolic energy space, i.e. 
\eqn{regas}
$$
u \in C^{0}(-T,0;L^{2}(\Omega))\cap L^{p}(-T,0;W^{1,p}(\Omega))\,,
$$
and it is a distributional solution to \rif{maineq} in the sense that
\eqn{weakp}
$$
-\int_{\Omega_T} u \varphi_t \, dx \, dt + \int_{\Omega_T} \langle a(x,t,Du), D\varphi\rangle\, dx\, dt =  \int_{\Omega_T} \varphi\, d \mu
$$
holds whenever $\varphi \in C^{\infty}_{c}(\Omega_T)$. In view of the available approximation theory we shall assume that
$\mu \in L^1$
without loss of generality, while upon letting $\mu\lfloor_{\er^{n+1} \setminus \Omega_T}=0$, we shall finally consider the case
\eqn{addmes}
$$
\mu \in L^1(\er^{n+1})\,.
$$
Assumptions \rif{regas} and \rif{addmes} will be then finally removed in Section \ref{generalmeasure}. There we shall deal with solutions to general measure data problems, where both \rif{regas} and \rif{addmes} are not any longer in force. In other words, we purse the usual path that consists of first deriving a priori estimates for more regular problems and solutions, and then recovering the general case by approximation. Notice that under the assumptions \rif{regas} and \rif{addmes} by standard density arguments the identity in \rif{weakp} remains valid whenever
$
\varphi \in W^{1,p}(\Omega_T) \cap L^\infty(\Omega_T) %%%TUOMO : changed %%%ROS rechanged
$
has a compact support.
\end{remark}

\subsection{Intrinsic and explicit Riesz potential estimates} Very recently, in \cite{KMcras, KMarma, mis3} for the case $p\geq 2$ and in \cite{DM2} for the subquadratic one, it has been shown that, surprisingly enough, the regularity theory of possibly degenerate quasilinear elliptic equations of the type $$-\divo\, a(Du)=\mu$$ completely reduces to that of standard Poisson equation
\eqn{poisson}
$$-\triangle u = \mu$$ up to the $C^{1}$-level, i.e.~up to the gradient continuity. Moreover, in some sense the regularity theory {\em can be actually linearized via Riesz potentials}. In particular, the gradient of solutions can be pointwise bounded via classical Riesz potentials exactly as it happens for solutions to \rif{poisson}, i.e., ~the inequality
\eqn{mainestg}
$$
|Du(x_0)|^{p-1} \leq c{\bf I}_1^{|\mu|}(x_0,r) + c\left(\mean{B(x_0,r)}(|Du|+s)\, dx\right)^{p-1}
$$
holds for a.e. point $x_0$, where
$$
{\bf I}_{1}^{|\mu|}(x_0,r):=\int_0^r \frac{|\mu|(B(x_0,\varrho))}{\varrho^{n-1}}\, \frac{d\varrho}{\varrho}
$$
denotes the standard truncated Riesz potential of $|\mu|$.
Our aim is to build a related theory for general degenerate parabolic problems of the type in \rif{modello} and \rif{maineq}. The main challenge here is to match the anticipated %%%TUOMO : wish -> ant...
a priori Riesz potential estimate with the inhomogeneous nature of equations such as \rif{modello}; it will be indeed part of the work to find {\em the proper formulation, suited to the geometry of the equations considered,} making this possible. We also remark that, even when applied to the stationary case, our results turn out to be more general than those contained in \cite{KMcras, KMarma} since the equations considered here are also allowed to have coefficients. The ultimate outcome is that, once again, the (spatial) gradient regularity theory of solutions to \rif{maineq} can be unified in a natural way with the one of the usual heat equation \rif{heateq}.
The analysis here unavoidably involves the concept of the {\em intrinsic geometry}, introduced and widely employed by DiBenedetto \cite{DBbook, Urbano, Baroni}. According to this principle, the lack of scaling (for $p\not=2$) of equations as
\eqn{heatp}
$$
u_t- \divo\, (|Du|^{p-2}Du)=0
$$ can be locally rebalanced by performing the regularity analysis of the solution on certain special cylinders adapted to the solution itself, indeed called {\em intrinsic parabolic cylinders}. More precisely, instead of using standard parabolic cylinders
\eqn{classicalc}
$$
Q_{r}(x_0,t_0):= B(x_0,r)\times (t_0-r^2, t_0)\,,
$$ one uses {\em cylinders whose time-length is stretched accordingly to the size of the gradient on the cylinder itself}. In other words, one is lead to consider cylinders of the type \eqn{intrinsiccc}
$$
Q_{r}^{\lambda}(x_0,t_0):= B(x_0,r)\times (t_0-\lambda^{2-p}r^2, t_0)\,, \qquad \lambda >0 \,,
$$
on which it simultaneously happens that a condition of the type
\eqn{loclink}
$$
\mean{Q_{r}^{\lambda}(x_0,t_0)} |Du|\, dx \, dt \lesssim \lambda
$$
is satisfied.
The use of the word intrinsic stems from the very basic fact that the parameter $\lambda$ appears on both sides of \rif{loclink}.
Ultimately, this has the effect of rebalancing the local anisotropic character of the equation allowing for proving {\em homogeneous regularity estimates:} in some sense, the equation \rif{heatp} looks like the heat equation when considered on $Q_{r}^{\lambda}(x_0,t_0)$. For instance, when considering standard parabolic cylinders, for solutions to \rif{heatp} it is only possible to prove bounds of the type
\eqn{di1}
$$
\sup_{Q_{r/2}(x_0,t_0)} |Du| \leq c(n,p) \mean{Q_{r}(x_0,t_0)} (|Du|+s+1)^{p-1}\, dx\, dt\,,
$$
whose lack of homogeneity precisely reflects that of the equation. In this sense the previous estimate is natural. When instead considering intrinsic cylinders with \rif{intrinsiccc}-\rif{loclink} being in force, {\em estimates become dimensionally homogeneous}:
\eqn{di2}
$$
c(n,p)\left(\mean{Q_{r}^\lambda(x_0,t_0)} (|Du|+s)^{p-1} \, dx \, dt\right)^{1/(p-1)} \leq \lambda  \Longrightarrow |Du(x_0,t_0)| \leq \lambda\,.
$$
Both \rif{di1} and \rif{di2} are basic results of DiBenedetto \cite{DBbook} while we just remark that intrinsic geometries are nowadays a basic and common tool to treat degenerate parabolic equations \cite{AM, KL1, KL2}.

The previous considerations, together with \rif{di1}-\rif{di2}, are actually the starting point for proving the desired potential estimates. Let us see how. Beside the usual caloric (truncated) Riesz potentials built upon standard parabolic cylinders as in \rif{classicalc}, that is
\eqn{classicalr}
$$
 {\bf I}_{\beta}^\mu(x_0,t_0;r):=
 \int_0^{r} \frac{|\mu|(Q_\varrho(x_0,t_0))}{\varrho^{N-\beta}}\, \frac{d\varrho}{\varrho}\,,\qquad  \qquad 0< \beta \leq N:=n+2\,,
$$
we introduce the {\em intrinsic Riesz potentials} as
\eqn{pari}
$$
 {\bf I}_{\beta, \lambda}^\mu(x_0,t_0;r):=
 \int_0^{r} \frac{|\mu|(Q_\varrho^{\lambda}(x_0,t_0))}{\varrho^{N-\beta}}\, \frac{d\varrho}{\varrho}\,,
$$
where $N$ is called, as usual, the parabolic dimension. Note that in such a way we have $ {\bf I}_{\beta}^\mu(x_0,t_0;r)\equiv  {\bf I}_{\beta,1}^\mu(x_0,t_0;r)$. At this stage the word ``intrinsic" merely refers to the fact that the additional parameter $\lambda$ has been considered in the definition in \rif{pari}, while at the moment no local linkage with solutions of the type in \rif{loclink} has been considered yet. This will come in a few moments: indeed, the right way to give an intrinsic formulation of the linear potential bounds is inspired by \rif{di2} and it is given in the following:
\begin{theorem}[Intrinsic Riesz potential bound]\label{main1} Let $u$ be a solution to \trif{maineq} under the assumptions \trif{asp}-\trif{aspd} and \trif{addmes}. There exist a constant $c > 1$ and a radius $R_0>0$, both depending only on $n,p,\nu,L, \omega(\cdot)$, such that
the following implication holds:
%\begin{eqnarray}
%&&\label{unoimpli}
\eqn{unoimpli}
$$
\begin{array}{c}
\displaystyle
 c{\bf I}_{1,\lambda}^\mu(x_0,t_0;r)+ c\left(\mean{Q_r^{\lambda}(x_0,t_0)} (|Du|+s)^{p-1} \, dx \, dt\right)^{1/(p-1)}   \leq \lambda \\ [10 pt]
 \hspace{7cm} \Longrightarrow |Du(x_0,t_0)| \leq \lambda
\end{array}
$$%\end{eqnarray}
whenever $Q_r^{\lambda}(x_0,t_0) \subset \Omega_T$, $(x_0,t_0)$ is Lebesgue point of $Du$, and $r \leq R_0$. When the vector field $a(\cdot)$ is independent of $x$, no restriction occurs on $r$, i.e., ~$R_0=\infty$.
\end{theorem}
Note that, as in a sense it was a priori required, \rif{unoimpli} allows to recover \rif{di2} when $\mu=0$; this is a first sign of the fact that \rif{unoimpli} is the ``correct intrinsic extension" of \rif{mainestg}. As a matter of fact Theorem \ref{main1} implies a gradient linear potential estimate involving standard Riesz potentials. Surprisingly enough, this is of the same type as the one which holds for the standard heat equation; moreover, when $\mu=0$, this reduces \rif{di1}. We indeed have the following:
\begin{theorem}[Riesz potential bound in classic form]\label{main3} Let $u$ be a solution to \trif{maineq}
in $\Omega_T$ under the assumptions \trif{asp}-\trif{aspd} and \trif{addmes}. There exists a constant $c$, depending only on $n,p,\nu,L,\omega(\cdot)$, such that
$$
|Du(x_0,t_0)|  \leq  c{\bf I}_{1}^{|\mu|}(x_0,t_0;r)+ c\mean{Q_{r}(x_0,t_0)} (|Du|+s+1)^{p-1} \, dx \, dt  % \label{appi}
$$
holds whenever $(x_0,t_0) \in \Omega_T$ is a Lebesgue point of $Du$ and whenever $Q_{r}(x_0,t_0)\subset \Omega_T$ is a standard parabolic cylinder such that $r \leq R_0$; here $R_0$ is the radius introduced in Theorem \ref{main1}.\end{theorem}
%\begin{remark}\label{nox} Both in Theorem \ref{main1} and \ref{main3} we can drop the restriction $r \leq R_0$ when the vector field $a(\cdot)$ is independent of $x$. This follows from a careful reading of the proofs. Measurable dependence on $t$ is still allowed without requiring $r \leq R_0$.
%\end{remark}
For the case $p=2$ the previous estimate recovers the main parabolic result in \cite{DM}. An immediate consequence of Theorem \ref{main3} is the following global bound via classical, non-truncated caloric Riesz potentials:
\begin{cor}\label{globalb} Let $u$ be a weak solution to the equation
\eqn{senzat}
$$
u_t - \divo \,a(t,Du)=\mu \qquad \mbox{in}\ \er^{n+1}
$$
under the assumptions \trif{asp}; moreover, assume that the global integrability $u \in L^{p-1}(-\infty,t_0;W^{1,p-1}(\er^n))$ holds for $t_0 \in \er$. There exists a constant $c$, depending only on $n,p,\ratio$, such that the upper bound
$$
|Du(x_0,t_0)| \leq c\int_{\{ t< t_0\}}\frac{d|\mu|(x,t)}{d_{\rm par}((x,t),( x_0, t_0 ))^{N-1}} %\qquad S_{t_0}:=\{(x,t)\, : \, t< t_0\}
$$
holds whenever $(x_0,t_0)$ is a Lebesgue point of $Du$.
\end{cor}
\begin{remark} In Theorem \ref{globalb} $d_{\rm par}(\cdot)$ denotes the standard parabolic distance in $\er^{n+1}$, which is defined by
%\eqn{pd}
$$
d_{\rm par}((x,t),(x_0,t_0)):=\max\left\{|x-x_0|, \sqrt{|t-t_0|}\right\}\,.
$$
The previous result shows that Theorems \ref{main1} and \ref{main3} play the role of the usual representation formulae via heat kernels for solutions to the heat equation. In recent years there has been a large activity devoted to the understanding the extent to which heat kernel estimates are still valid when passing to more general settings, as for instance Lie groups and manifolds \cite{DGS, SC}. In this paper we are interested, in a dual but yet related way, to see the extent to which estimates as those implied by well-behaving heat kernels can be recovered in the nonlinear degenerate setting. Our results also connects to a recent fact observed in \cite{LM}, and concerning the $p$-superharmonicity of linear Riesz potentials.
\end{remark}
The proof of Theorem \ref{main1} opens the way to an optimal continuity criterion for the gradient still involving only classical Riesz potentials and that, as such, is again independent of $p$.
\begin{theorem}[Gradient continuity via linear potentials]\label{mainc1} Let $u$ be a solution to \trif{maineq} % such that $Du$ is continuous
in $\Omega_T$ under the assumptions \trif{asp}-\trif{aspd} and \trif{addmes}. If
$$
\lim_{r \to 0}{\bf I}_{1}^\mu(x,t;r)  = 0$$ locally uniformly w.r.t. $(x,t)$, then $Du$ is continuous in $\Omega_T$.
\end{theorem}
An important corollary involves Lorentz spaces:
\begin{cor}[Lorentz spaces criterion]\label{mainl} Let $u$ be a solution to \trif{maineq} % such that $Du$ is continuous
in $\Omega_T$ under the assumptions \trif{asp}-\trif{aspd}. If $\mu \in L(N,1)$, that is if
$$
\int_0^{\infty} |\{(x,t) \in \Omega_T \, : \, |\mu(x,t)| > \lambda \}|^{1/N} d \lambda < \infty\,,
$$
then $Du$ is continuous in $\Omega_T$.
\end{cor}
Corollary \ref{mainl} substantially improves the ones in \cite{KMr} claiming only the boundedness of the gradient under the assumption $\mu \in L(N,1)$; see also \cite{CM} for a related global boundedness result in the elliptic case. It might be interesting to note how the above result naturally extends to the parabolic case the classical gradient continuity results valid in the elliptic case, starting from those available for the Poisson equation $-\triangle u=\mu$ in domain of $\er^n$. For this it is known that the condition $\mu \in L(n,1)$ is a sufficient one for the gradient continuity. This is in turn related to, and indeed implied by, a classical result of Stein \cite{Steinc} that claims the continuity of a function $f$ whenever its distributional derivatives belong to $L(n,1)$. Corollary \ref{mainl} gives the precise nonlinear parabolic analog of such classical facts. As expected, the space dimension $n$ is replaced by the parabolic one $N=n+2$, which is naturally associated to the standard parabolic metric.

Preliminary to the proof of the continuity criterion, there is another result which claims the VMO gradient regularity under weaker assumptions on the measure $\mu$.
\begin{theorem}[VMO gradient regularity]\label{mainv1} Let $u$ be a solution to \trif{maineq} under the assumptions \trif{asp}-\trif{aspd} and \trif{addmes}. If
$
{\bf I}_{1}^\mu(x,t;r)$ is locally bounded in $\Omega_T$ and if
\eqn{assvmo}
$$
\lim_{r \to 0}\frac{|\mu|(Q_r(x,t))}{r^{N-1}}=0
$$
locally uniformly in $\Omega_T$ w.r.t. $(x,t)$, then $Du$ is locally $VMO$ in $\Omega_T$, that is
$$
\lim_{R \to 0} \, \sup_{r \leq R, Q_{r} \subset Q'}  \, \mean{Q_r} |Du - (Du)_{Q_r}|\, dx\, dt =0
$$
for every open subset $Q' \Subset \Omega_T$.
\end{theorem}
\subsection{General measure data problems}\label{generalmeasure} Solutions to measure data problems are usually found by approximation procedures via solutions to more regular problems. These are of the type
\eqn{appr}
$$
(u_h)_t-\divo\, a(x,t,Du_h)=\mu_h \in L^{\infty}\,, \qquad h \in \en\,, $$
where \rif{regas} holds for $u_h$ and $\mu_h$ is usually a convolution of $\mu$. The point is that solutions to measure data problems do not belong, in general, to the energy space. This section is also aimed at justifying that we may actually work under the {\em apparently additional assumption}~\rif{regas}. More precisely, the exact definition of SOLA, is given in the following:
\begin{definition}[\cite{B, KLPj, KLP}]\label{soladef} A SOLA (Solution Obtained as a Limit of Approximations) 
to \trif{maineq} is a distributional solution $u \in L^{p-1}(-T, 0;W^{1,p-1}(\Omega))$ to \trif{maineq} in $\Omega_T$, such that $u$ is the limit of solutions 
$$
u_h \in C^{0}(-T,0;L^{2}(\Omega))\cap L^{p}(-T,0;W^{1,p}(\Omega))
$$
to equations as \rif{appr} in the sense that
$
u_h \to u
$ in $L^{p-1}(-T,0;W^{1,p-1}(\Omega))$, $L^\infty \ni \mu_h \to \mu$ weakly in the sense of measures and such that
\eqn{convergencemeasures}
$$\limsup_{h} \,|\mu_h|(Q)\leq |\mu|(\lfloor Q\rfloor_{\rm par})$$
for every cylinder $Q =  B \times (t_1, t_2)  \subseteq \Omega_T$, where $B \subset \Omega$ is a bounded open subset.
\end{definition}
In the right hand side of \rif{convergencemeasures} appears the symbol $\lfloor Q\rfloor_{\rm par}$, which denotes the parabolic closure of $Q$ defined in \rif{parclo} below. For more on this kind of solutions see Remark \ref{solaremark} below; in particular, requiring \rif{convergencemeasures} is neither unnatural nor restrictive.
Our estimates remain valid for SOLA and, in fact, the following holds:
\begin{theorem}\label{solath} The statements of Theorems \ref{main1}, \ref{main3}, \ref{mainc1} and \ref{mainv1} continue to hold for SOLA $u \in L^{p-1}(-T,0;W^{1,p-1}(\Omega))$ to \trif{maineq}, under the only assumptions \trif{asp}-\trif{aspd}. As a consequence, the results in Corollaries \ref{globalb} and \ref{mainc1}  hold for SOLA as well. \end{theorem}
\subsection{Comparison with nonlinear estimates}\label{compi}
Theorem \ref{main1} improves the previously potential estimates via nonlinear potentials \cite{KMW}, bringing them {\em to the desired optimal level}. Based on elementary dimension analysis we conjecture that the result of Theorem \ref{main1} cannot be improved by the use of any other nonlinear potential. Theorem \ref{main1} is optimal also with respect to the regularity assumed on the coefficients dependence $x \mapsto a(x, \cdot)$, that is \rif{aspd}. Indeed, already in the linear elliptic case $$\divo\, (\tilde a(x)Du)=0\,,$$ Dini-continuity of elliptic coefficients matrix $\tilde a(x)$ is essential in order to get gradient boundedness. Merely continuous coefficients are not sufficient to ensure that the gradient belongs even to BMO, see \cite{maz}. 

Now, let us see how Theorem \ref{main1} improves the previously known estimates via nonlinear potentials. In \cite{KMW} a Wolff potential type gradient bound has been obtained for equations without coefficients of the type $u_t- \divo\, a(Du)=\mu$. More precisely, in \cite{KMW} we introduced the following {\em intrinsic Wolff potentials}:
$$
{\bf W}_{\lambda}^\mu (x_0,t_0;r):=\int_0^{r} \left(\frac{|\mu|(Q^{\lambda}_\varrho(x_0,t_0))}{\lambda^{2-p} \varrho^{N-1}} \right)^{1/(p-1)} \, \frac{d\varrho}{\varrho}\,,\qquad N=n+2 \qquad \lambda >0\,.
$$
See \cite{HM, HW, KM} for more on Wolff potentials. We then proved the existence of a universal constant $c_w \equiv c_w(n,p,\ratio)$ for which
\eqn{gen}
$$
\begin{array}{c}
\displaystyle c_w{\bf W}_{\lambda}^\mu (x_0,t_0;r)+c_w\left(\mean{Q_r^\lambda(x_0,t_0)} (|Du|+s)^{p-1} \, dx \, dt\right)^{1/(p-1)} \leq \lambda \\  [10 pt]\hspace{4cm} \Longrightarrow |Du(x_0,t_0)| \leq \lambda
\end{array}
$$
holds.
Let us now show that \rif{unoimpli} implies \rif{gen}, for
\eqn{cw}
$$c_w := \left[4^{N+1}(\log 2)^{2-p}c\right]^{1/(p-1)}+4^{N}c$$ and radii $r \leq R_0$, where $c$ and $R_0$ are the constants appearing in the statement of Theorem \ref{main1} (no restriction on radii appears in the case of equations as in \rif{senzat}). With $r_i=r/2^i$ for integers $i \geq 0$, H\"older's inequality for series (as it is $p \geq 2$ here), gives
\begin{eqnarray*}
{\bf I}_{1,\lambda}^\mu(x_0,t_0;r/2) & = &\sum_{i=1}^\infty \int_{r_{i+1}}^{r_i} \frac{|\mu|(Q^{\lambda}_{\varrho}(x_0,t_0))}{ \varrho^{N-1}} \frac{d\varrho}{\varrho}\\ &\leq & 2^{N-1}\log 2 \sum_{i=1}^\infty \frac{|\mu|(Q^{\lambda}_{r_i}(x_0,t_0))}{r_i^{N-1}} \\
&\leq & 2^{N-1}\log 2\left[\sum_{i=1}^\infty \left(\frac{|\mu|(Q^{\lambda}_{r_i}(x_0,t_0))}{r_i^{N-1}}\right)^{1/(p-1)} \right]^{p-1}\\
%&= & 4^{N}(\log 2)^{2-p}\left[\sum_{i=1}^\infty \left(\frac{|\mu|(Q^{\lambda}_{r_i}(x_0,t_0))}{r_{i-1}^{N-1}}\right)^{1/(p-1)} \int_{r_{i}}^{r_{i-1}}\frac{d\varrho}{\varrho} \right]^{p-1}\\
&\leq & 4^{N}(\log 2)^{2-p}\left[\sum_{i=1}^\infty  \int_{r_{i}}^{r_{i-1}}\left(\frac{|\mu|(Q^{\lambda}_{\varrho}(x_0,t_0))}{\varrho^{N-1}}\right)^{1/(p-1)}\frac{d\varrho}{\varrho} \right]^{p-1}\\
&= & 4^{N}(\log 2)^{2-p}\lambda^{2-p}\left[{\bf W}_{\lambda}^\mu (x_0,t_0;r)\right]^{p-1}\\
&\leq & 4^{N}(\log 2)^{2-p}c_w^{1-p}\lambda \leq \frac{\lambda}{2c}\,.
\end{eqnarray*}
Notice that to derive the second-last estimate we have use the inequality in the first line of \rif{gen}, while in the last estimate we have used \rif{cw}. Using the inequality in the last display, again \rif{cw} and finally the left hand hand side of \rif{gen}, a standard manipulation gives
$$
c{\bf I}_{1,\lambda}^\mu(x_0,t_0;r/2)+ c\left(\mean{Q_{r/2}^{\lambda}(x_0,t_0)} (|Du|+s)^{p-1} \, dx \, dt\right)^{1/(p-1)}   \leq \lambda
$$
so that the right hand side inequality in \rif{gen}, that is $|Du(x_0,t_0)| \leq \lambda$, follows applying Theorem \ref{main1}.

The improvement from \rif{gen} to \rif{unoimpli} is rather strong both from the viewpoint of the theoretical significance - as now the regularity theory of quasilinear equations is unified with that of the heat equation up to spatial gradient continuity - and from the one of the consequences. For instance, when looking for sharp criteria for establishing $Du \in L^{\infty}$ (and eventually in $C^0$) in terms of Lorentz spaces, the result in \rif{gen} gives that
\eqn{condi1} 
$$
\mu \in L(N,1/(p-1)) \Longrightarrow Du \in L^\infty  \,,
$$
where we recall that
$$
\mu \in L(N,1/(p-1)) \Longleftrightarrow \int_0^{\infty} \lambda^{\frac{2-p}{p-1}}|\{(x,t) \in \Omega_T \, : \, |\mu(x,t)| > \lambda \}|^{\frac{1}{N(p-1)}} d \lambda < \infty\,.
$$
The condition of Corollary \ref{mainl} is clearly stronger that the one in \rif{condi1}, as $L(N,1)\subset L(N,1/(p-1))$, this inclusion being strict for $p>2$. For further properties of Lorentz spaces we refer for instance to \cite{steinweiss}. Moreover, it is easy too see that more refined criteria in terms of density/concentration are provided by \rif{unoimpli} with respect to \rif{gen} when $\mu$ is genuinely a measure. We also remark that Wolff potentials play a major role in the analysis of the fine properties of quasilinear equations (see for instance \cite{KMp, KM, PV1, PV2, TW}); since the estimates contained in this paper are stronger than those involving Wolff potentials, we expect they will have a similar, if not stronger, impact in the future. 
\subsection{Techniques}
Finally, a few comments on the methods used in this paper. Several new ingredients are needed to deal with the parabolic case with respect to the previous elliptic one \cite{KMl}, and the proofs depart considerably from those proposed before. The proof of Theorem \ref{main1} involves a very delicate, double-step induction procedure based on a few ingredients that re-shuffle, in a pointwise manner, some classical methods used in linear Calder\'on-Zygmund theory and combine them with the use of intrinsic geometry. Extensive use of nonlinear potential theoretic methods and regularity theory is made throughout. Let us briefly describe the heuristic used here by specializing for simplicity to the model case \rif{modello} and considering $\mu \in L^1$; the essence relies in careful procedure that allows to ``linearize the equation" and control the possible degeneracy in a precisely quantified way at every scale. The following argument will be purely formal. We consider a dyadic shirking sequence of intrinsic cylinders for $i \geq 0$
$$
\ldots Q_{r_{i+1}}^\lambda(x_0,t_0) \subset Q_{r_{i}}^\lambda(x_0,t_0) \subset Q_{r_{i-1}}^\lambda(x_0,t_0)\ldots \qquad r_i:= \sigma^ir
$$
where $\sigma \in (0,1)$ is a constant depending only on $n,p$ and $\lambda$ is as in \rif{unoimpli}. A suitable exit time argument, together with very careful regularity estimates for solutions to homogeneous equations, gives
\eqn{lowerbounds}
$$
\lambda - \mbox{``quantified error"} \lesssim |Du|  \quad \mbox{on} \ Q_{r_{i}}^\lambda(x_0,t_0) \  \ \mbox{for $i$ large enough}\,.
$$
This is something that in a way we can always assume, starting from an exit time index, otherwise we are going to get an opposite inequality for the integral averages $|(Du)_{Q_{r_{i}}^\lambda(x_0,t_0)}| \lesssim \lambda$, that eventually leads to an immediate proof of $|Du(x_0,t_0)| \leq \lambda $, and therefore of \rif{unoimpli}. Assuming \rif{lowerbounds} leads to implement a delicate iteration procedure whose finally outcome is the following inequality:
\eqn{boundheu}
$$
\mean{Q_{r_{j}}^\lambda(x_0,t_0)} |Du|^{p-1}\, dx \, dt \lesssim \lambda^{p-1}$$
that again implies \rif{unoimpli}. Note that proving \rif{boundheu} not only allows to implement the iteration but also allows  to use, at each scale, the intrinsic geometry methods (compare with \rif{loclink}). As emphasized, the key to the proof of Theorem \ref{main1} is the lower bound in \rif{lowerbounds}; let us now give a formal but yet convincing argument showing how a condition as \rif{lowerbounds} allows to get \rif{unoimpli} and why intrinsic Riesz potentials and and conditions as \rif{unoimpli} naturally occur. Let us consider then \rif{lowerbounds} to be satisfied on $Q_{r}^{\lambda}(x_0,t_0)$ with a null error, i.e.,  $\lambda \lesssim |Du|$, and let us assume the first inequality in \rif{unoimpli}. The lower bound $\lambda \lesssim |Du|$ in turn allows to gain coercivity enough to treat the equation in display \rif{modello} as a heat equation with a coefficient, that is
$
u_t- \divo\, (\lambda^{p-2}Du)=\mu
$
that we can rewrite as
$$
\lambda^{2-p} u_t- \triangle u= \lambda^{2-p}\mu \qquad \mbox{in} \ Q_{r}^\lambda(x_0,t_0)\,.
$$
Now the effect of the use of intrinsic geometry and of the intrinsic Riesz potential shows-up. Changing variables and introducing
$$
v(x,t):= \frac{u(x_0+rx, t_0+\lambda^{2-p}r^2t)}{r}\,, \quad \tilde \mu(x,t):= \lambda^{2-p}r\mu(x_0+rx, t_0+\lambda^{2-p}r^2t)
$$
for $(x,t) \in Q_1= B_1\times (-1,0)$, we have
\eqn{heu0}
$$v_t- \triangle v= \tilde\mu\,.$$ Next, we apply the standard Riesz potential bound for solutions to \rif{heu0}, that is
\eqn{heu1}
$$
|Dv(0,0)|  \lesssim {\bf I}_{1}^{|\tilde \mu|}(0,0;1)+ \left(\mean{Q_{1}(0,0)} |Dv|^{p-1} \, dx \, dt\right)^{1/(p-1)} \,. % \label{appi}
$$
Changing variables back to $\mu$ we notice
\begin{eqnarray*}{\bf I}_{1}^{|\tilde \mu|}(0,0;1) &= &
\lambda^{2-p} r \int_0^1 \mean{Q_\varrho(0,0)} |\mu(x_0+rx, t_0+\lambda^{2-p}r^2t)|\, dx\, dt \, d\varrho\\
&= &  \lambda^{2-p} r \int_0^{1} \mean{Q_{\varrho r}^\lambda(x_0,t_0)} |\mu(x,t)|\, dx\, dt \, d\varrho \\
&= &  \lambda^{2-p} \int_0^{r} \mean{Q_{\varrho}^\lambda(x_0,t_0)} |\mu(x,t)|\, dx\, dt \, d\varrho =   \lambda^{2-p} {\bf I}_{1, \lambda}^\mu(x_0,t_0;r) \lesssim \lambda\,,
\end{eqnarray*}
where in the last inequality we have used the first line in \rif{unoimpli}. Finally, scaling back to $u$, using the previous inequality together with the first line of \rif{unoimpli}, we conclude with
$
|Du(x_0,t_0)| = |Dv(0,0)| \lesssim \lambda$, that is the proof of \rif{unoimpli}. The one outlined in the last lines is only a heuristic argument used to show how intrinsic Riesz potentials play a decisive and natural role in this context, but its rigorous implementation is highly nontrivial and involves a refined %%%TUOMO : very delicate
double induction argument that exploits rather subtle aspects of regularity theory of degenerate parabolic equations. Several tools are used here. One of the main points is that the analysis of the relevant iterating quantities must be performed at two different levels, using different energy spaces. Indeed, since we are  dealing essentially with measure data problems, the natural spaces involved are larger than $L^{p}(-T,0;W^{1,p})$. This, together with the lack of reverse H\"older type inequalities and homogeneous estimates which is typical when dealing with degenerate parabolic equations, reflects in a simultaneous use of two different spaces, namely $L^{1}(-T,0;W^{1,1})$ and $L^{p-1}(-T,0;W^{1,p-1})$. Eventually, a very delicate interplay between local regularity of solutions to homogeneous equations and comparison estimates must be exploited in the framework of intrinsic geometries thanks to exit time arguments and the use of intrinsic Riesz potentials. The proof of Theorem \ref{main1} eventually opens the way to the continuity analysis and in particular to Theorem \ref{mainc1}. For this we shall readapt the iteration procedure of Theorem \ref{main1} to estimate oscillations rather than the size of the gradient. This in turn imposes to consider a priori infinite many exit times arguments used to control the degeneracy of the equation via the oscillations of the gradient, and vice-versa.

\section{Preparations}
\subsection{General notation} In what follows we denote by $c$ a general positive constant, possibly varying from line to line; special occurrences will be denoted by $c_1, c_2, \bar c_1, \bar c_2$ or the like. All these constants will always be larger or equal than one; moreover relevant
dependencies on parameters will be emphasized using parentheses, i.e., ~$c_{1}\equiv c_1(n,p,\ratio)$ means that $c_1$ depends only on $n,p,\ratio$. We denote by $B(x_0,r):=\{x \in \er^n \, : \,  |x-x_0|< r\}$ the open ball with center $x_0$ and radius $r>0$; when not important, or clear from the context, we shall omit denoting the center as follows: $B_r \equiv B(x_0,r)$. Unless otherwise stated, different balls in the same context will have the same center. We shall also denote $B \equiv B_1 = B(0,1)$
if not differently specified. 
In a similar fashion standard and intrinsic parabolic cylinders with vertex $(x_0,t_0)$ and width $r>0$ have been defined in \rif{classicalc} and \rif{intrinsiccc}, respectively. When the vertex will not be important in the context or it will be clear that all the cylinders occurring in a proof will share the same vertex, we shall omit to indicate it, simply denoting $Q_r$ and $Q_r^\lambda$ for the cylinders in \rif{classicalc} and \rif{intrinsiccc}, respectively. We recall that if $Q = \mathcal A \times (t_1, t_2)$ is a cylindrical domain, the usual parabolic boundary of $Q$ is
$\partial_{\rm par} Q:= (\mathcal A \times \{t_1\}) \cup (\partial \mathcal A \times [t_1,t_2))$,
and this is nothing else but the standard topological boundary without the upper cap $\bar{\mathcal A} \times \{t_2\}$. Accordingly, we define the parabolic closure of $Q$ as \eqn{parclo}
$$
\lfloor Q\rfloor_{\rm par} := Q \cup \partial_{\rm par} Q\,.
$$
With $\mathcal O \subset \er^{n+1}$ being a measurable subset with positive measure, and with $g \colon \mathcal O \to \er^n$ being a measurable map, we shall denote by  $$
   (g)_{\mathcal O} \equiv \mean{\mathcal O}  g(x,t) \, dx \, dt := \frac{1}{|\mathcal O|}\int_{\mathcal O}  g(x,t) \, dx\, dt
$$
its integral average; here $|\mathcal O|$ denotes the Lebesgue measure of $\mathcal O$. A similar notation is adopted if the integral is only in space or time.
 In the rest of the paper we shall use several times the following elementary property of integral averages:
\eqn{mediaprop}
$$
\left(\mean{\mathcal O} |g-(g)_{\mathcal O}|^{q}\, dx \, dt\right)^{1/q}\leq 2\left(\mean{\mathcal O} |g-\gamma|^{q}\, dx \, dt\right)^{1/q}\,,
$$
whenever $\gamma \in \er^n$ and $q \geq 1$. The oscillation of $g$ on $\mathcal O$ is instead defined as
$$
\osc_{\mathcal O}\, g := \sup_{(x,t),(\tilde x, \tilde t)\in \mathcal O}|g(x,t)-g(\tilde x,\tilde t)|\,.
$$
Finally, we remark that we shall denote the partial derivative with respect to time of a function $u$ both by $u_t$ and by $\partial_t u$; moreover, the letter $\lambda$ will always denote a positive number. Further relevant notation is at the beginning of the next section.
\subsection{Comparison maps} \label{comp maps} %%%ROS rewording in the following lines
The basic setup in this section is tailored to the needs of the proof of Theorem \ref{main1} and subsequent results. Therefore we shall consider $u$ to be an energy solution to \trif{maineq} under the assumptions \trif{asp}-\trif{aspd} and \trif{addmes} until the end Section \ref{compi2}; only in Section \ref{SOLA} we shall discuss the general case, thereby treating SOLA and discarding assumption \rif{addmes}. With a point $(x_0,t_0) \in \Omega_T$ being fixed, and given an intrinsic cylinder of the type $$Q_{r}^{\lambda}(x_0,t_0) \equiv B(x_0,r) \times (t_0 - \lambda^{2-p} r^2,t_0)$$
such that $Q_{2r}^{\lambda}(x_0,t_0) \subset \Omega_T$, we consider a  family of nested parabolic cylinders
\eqn{shri}
$$
Q_j \equiv B_j \times T_j \equiv B(x_0,r_j) \times (t_0 - \lambda^{2-p} r_j^2,t_0) \subseteq \Omega_T \,, \qquad r_j := \sigma^j r\,,
$$
for a fixed decay parameter $\sigma \in (0,1/4)$; notice that we always have
$$
\ldots Q_{j} \subset \frac 14 Q_{j-1} \subset Q_{j-1}\ldots
$$
Accordingly, we consider their dyadic, parabolic dilations
\[
\tau Q_j \equiv Q_{\tau r_j}^{\lambda}(x_0,t_0)\equiv \tau  B_j \times \tau  T_j \equiv B(x_0,\tau r_j) \times (t_0 - \lambda^{2-p} (\tau r_j)^2,t_0)
\]
for $\tau >0$; notice that here, slightly abusing the notation, we are denoting $$\tau  T_j \equiv (t_0 - \lambda^{2-p} (\tau r_j)^2,t_0) \,.$$ A similar notation will occur several times in rest of the paper. Now, let
$$
   w_j \in C^0(T_j ; L^2(B_j))\cap L^p(T_j ; W^{1,p}(B_j))
$$
be the unique solution to the Cauchy-Dirichlet problem
\begin{equation}\label{CD-local}
    \begin{cases}
    \partial_t w_j-\divo\, a(x,t,Dw_j)=0&\mbox{in $ Q_j$}
    \\
    \hfill w_j=u& \mbox{on $\partial_{\rm par} Q_j$\,.}
    \end{cases}
\end{equation}
After having defined $w_j$, we also define
$$
   v_j \in C^0(\tfrac12 T_j ; L^2(\tfrac12 B_j))\cap L^p(\tfrac12 T_j ; W^{1,p}(\tfrac12 B_j))
$$
as the unique solution to the frozen Cauchy-Dirichlet problem
\begin{equation}\label{CD-local v}
    \begin{cases}
    \partial_t v_j-\divo\, a(x_0,t,Dv_j)=0 \ \ &\mbox{in $ \frac12 Q_j$}
    \\
    \hfill v_j=w_j& \mbox{on $\partial_{\rm par} \left(\frac12 Q_j\right)$\,.}
    \end{cases}
\end{equation}

\subsection{A priori estimates for comparison maps} We now derive various a priori estimates for $w_j$ and $v_j$, starting from $L^\infty$-bounds.
When turning our attention to $w_j$ we need to use the results recently established in \cite{KMjmpa}, that allow to deal with equations with non-constant, Dini-continuous spatial coefficients. We start with a statement in terms of intrinsic geometry.
\begin{theorem}[Intrinsic gradient bound] \label{thm:Dw(j) bounded 1}
Let $w_j$ be as in \eqref{CD-local}. There exists a positive radius $R_1\equiv R_1(n,p,\nu,L,\omega(\cdot))$ and
a constant $c_1\equiv c_1(n,p,\nu,L)$ such that if $\varrho \in (0,R_1)$ and
\[
% Q_r^{\lambda_0} \equiv
Q_\varrho^{\lambda_0}(x_1,t_1)  := B(x_1,\varrho) \times (t_1 -\lambda_0^{2-p}\varrho^2,t_1)  \subset Q_j\,,
\]
then the implication
\begin{equation} \nonumber \label{Dw(j) bounded 1}
    c_1 \left(\mean{Q_\varrho^{\lambda_0}(x_1,t_1)}(|Dw_j|+s)^p \,dx \,dt\right)^{1/p} \leq \lambda_0 \quad \Longrightarrow \quad |Dw_j(x_1,t_1)| \leq \lambda_0
\end{equation}
holds.
\end{theorem}
\begin{proof} This result has been proved and used in \cite[Theorem 1.1, Theorem 1.3, Theorem 4.1]{KMjmpa}. The proof is exactly the one given in the proof of \cite[Theorem 1.1]{KMjmpa}, once \cite[Lemma 4.3]{KMjmpa} is used instead of \cite[Lemma 4.2]{KMjmpa}, see also \cite[Remark 4.1]{KMjmpa}. We also remark that no restriction on $\varrho$ is necessary  when the vector field $a(\cdot)$ is independent of $x$.
\end{proof}
The pointwise bound of Theorem~\ref{thm:Dw(j) bounded 1} can be turned into an $L^\infty$-bound of exactly the same type proved by DiBenedetto \cite{DBbook} for equations with no coefficients, see also Theorem \ref{thm:Dv(j) bounded} below. Since we are going to cover the case of equations with measure data, where solutions with low degree of integrability naturally appear, we need to lower the $p$-integrability exponent to $(p-1)$ to get the correct form of a priori estimates. All this is done in the next
\begin{cor} \label{cor:Dw(j) bounded 2}
Let $w_j$ be as in \eqref{CD-local} and $R_1$ as in Theorem~\ref{thm:Dw(j) bounded 1}. There exists a constant
$c_2\equiv c_2(n,p,\nu,L)$ such that if $r \in (0,R_1)$, then
%\begin{equation}\label{Dw(j) bounded}
$$    \sup_{\tau_m Q_j} |Dw_j|+s \leq c_2 (\lambda +s)+ \frac{c_2 \lambda^{2-p}}{2^{n+2}(1-\tau_m)^{n+2}} \mean{Q_j} (|Dw_j|+s)^{p-1} \, dx \, dt
$$%\end{equation}
holds whenever $\tau_m \in (0,1)$.
In particular, we have
%\eqn{unmezzo}
$$    \sup_{\frac12 Q_j} |Dw_j|+s \leq c_2 (\lambda +s)+ c_2 \lambda^{2-p} \mean{Q_j} (|Dw_j|+s)^{p-1} \, dx \, dt\,.
$$
\end{cor}
\begin{proof}
Define
\[
\lambda_0\equiv \lambda_0(\tau, \tau') :=  \frac12 \sup_{\tau Q_j} |Dw_j| +\lambda +s  + \frac{2 c_1^p}{(\tau-\tau')^{n+2}} \lambda^{2-p} \mean{Q_j} (|Dw_j|+s)^{p-1} \, dx \, dt
\]
whenever $\tau_m \leq \tau' < \tau \leq 1$, where $c_1 \equiv c_1(n,p,\nu,L)$ is as in Theorem~\ref{thm:Dw(j) bounded 1}. As $\lambda_0 \geq \lambda$ and $p \geq 2$, we clearly have that for $\delta := \tau-\tau'$ the inclusion
\[
Q_{\delta r_j}^{\lambda_0}(x_1,t_1):=B(x_1,\delta r_j) \times (t_1 -\lambda_0^{2-p}(\delta r_j)^2,t_1) \subset \tau Q_j=\tau Q_{r_j}^{\lambda}(x_0,t_0)
\]
holds whenever $(x_1,t_1)\in  \tau' Q_j$.
Furthermore, by using the very definition of $\lambda_0$ we may estimate
\begin{eqnarray*}
\nonumber && c_1 \left(\mean{Q_{\delta r_j}^{\lambda_0}(x_1,t_1)}(|Dw_j|+s)^p \,dx \,dt\right)^{1/p}
\\ \nonumber  && \qquad \leq c_1 2^{1/p} \lambda_0^{1/p} \left(\mean{Q_{\delta r_j}^{\lambda_0}(x_1,t_1)}(|Dw_j|+s)^{p-1} \,dx \,dt\right)^{1/p}
\\ \nonumber  && \qquad \leq c_1 2^{1/p} \lambda_0^{1/p} \left(\frac{|Q_j|}{|Q_{\delta r_j}^{\lambda_0}(x_1,t_1)|} \right)^{1/p} \left(\mean{Q_j}(|Dw_j|+s)^{p-1} \,dx \,dt\right)^{1/p}
\\ \nonumber  && \qquad = c_1 2^{1/p} \lambda_0^{1/p} \left(\frac{\lambda^{2-p}}{\delta^{n+2} \lambda_0^{2-p}} \right)^{1/p} \left(\mean{Q_j}(|Dw_j|+s)^{p-1} \,dx \,dt\right)^{1/p}
\\ \nonumber  && \qquad \leq  c_1 2^{1/p} \lambda_0^{1/p} \left(\frac{\lambda^{2-p}}{\delta^{n+2} \lambda_0^{2-p}} \right)^{1/p} \left(\frac{\lambda _0 \delta^{n+2}}{2 c_1^p \lambda^{2-p}} \right)^{1/p} = \lambda_0\,.
\end{eqnarray*}
Therefore Theorem~\ref{thm:Dw(j) bounded 1} implies that $|Dw_j(x_1,t_1)| \leq \lambda_0$. But this holds for all $(x_1,t_1) \in \tau' Q_j$ and thus
\begin{equation}\nonumber \label{zadaa}
    \sup_{\tau' Q_j} |Dw_j| \leq \frac12 \sup_{\tau Q_j} |Dw_j| +\lambda + s + \frac{2 c_1^p \lambda^{2-p}}{(\tau-\tau')^{n+2}} \mean{Q_j} (|Dw_j|+s)^{p-1} \, dx \, dt
\end{equation}
follows. Lemma \ref{simpfun} below applied with $\varphi(\tau)= \sup_{\tau Q_j} |Dw_j|$ then concludes the proof by properly choosing the constant $c_2$.
\end{proof}
The nest result is a classical iteration lemma for the of which we refer to \cite[Lemma 6.1]{G}.
\begin{lemma}\label{simpfun}Let $\varphi: [\tau_m, 1]\to [0,\infty)$, with $\tau_m \in (0,1)$, be a function such that
$$
\varphi(\tau') \leq \frac12 \varphi(\tau) +K+ \frac{\B}{(\tau-\tau')^{n+2}}
\quad \mbox{holds for every}\ \ \tau_m \leq \tau' < \tau\leq 1\;,
$$
where $\B, K\geq 0$. Then $ \varphi(\tau_m) \leq
c(n)K+(1-\tau_m)^{-(n+2)}\B$.
\end{lemma}
Corollary \ref{cor:Dw(j) bounded 2} obviously holds for $v_j$ too, and in this case it is a by now classical estimate of DiBenedetto \cite{DBbook}, as already mentioned above. See also for example~\cite{AM, KMW} for similar bounds. We report the statement for completeness.
\begin{theorem} \label{thm:Dv(j) bounded}
Let $v_j$ be as in~\eqref{CD-local v}. For a constant  $c_3 \equiv c_3(n,p,\nu,L)$ we have
$$
    \sup_{\frac14 Q_j} |Dv_j|+s \leq c_3 (\lambda +s)+ c_3 \lambda^{2-p} \mean{\frac12Q_j} (|Dv_j|+s)^{p-1} \, dx \, dt\,.
$$
\end{theorem}
We now pass to give oscillation estimates for $w_j$ and $v_j$. The next result provides a gradient oscillations estimate for  solutions to homogeneous equations with Dini-continuous coefficients.
\begin{theorem}[Continuity estimate]\label{thm:Dw(j) continuous}
Let $w_j$ be as in \eqref{CD-local}, then $Dw_j$ is continuous. Moreover, assume that
\eqn{global}
$$
\sup_{\frac12 Q_j}|Dw_j| + s \leq A \lambda
$$
holds for some $A\geq 1$. Then, for any $\delta \in (0,1)$ there exists a positive constant $\sigma_1 \equiv \sigma_1(n,p,\nu,L,A,\delta, \omega(\cdot)) \in (0,1/4)$
such that
\begin{equation}\label{Dw(j) continuous}
\osc_{\sigma_1 Q_{j}} Dw_j \leq \delta \lambda\,.
\end{equation}
\end{theorem}
\begin{proof} The starting point of the proof is the work in \cite{KMjmpa}, where the continuity of the gradient of solutions to parabolic equations as in \rif{maineq} has been proved under the assumption of Dini-continuity of the space coefficients; this by the way immediately implies the continuity of $Dw_j$ claimed in the statement. What we need here is a quantitative bound on the oscillations of $Dw_j$. To this aim, let us briefly recall the main arguments in \cite[proof of Theorem 1.3]{KMjmpa}, where the continuity properties of $Du$ are formulated and proved in terms of the auxiliary vector field
$$
V(z)  := (|z|^2+s^2)^{(p-2)/4}z
$$
and the related field $V(Dw_j)$. For the use of such maps in the present context we refer to \cite{KMjmpa}; the only property we shall use here is the following inequality:
\eqn{eleV}
$$
|z_1-z_2|\leq c_v \frac{|V(z_1)-V(z_2)|}{(s^2+|z_1|^2+|z_2|^2)^{(p-2)/4}}\,,
$$
that holds for $c_v \equiv c_v(n,p)$ and for all vectors $z_1,z_2 \in \er^n$ which are not simultaneously null; see for instance \cite{sns}. By following the arguments developed for \cite[(5.15)]{KMjmpa} it can be proved that for every $\eps \in (0,1)$ there exists a positive radius $R_\eps \equiv R_\eps(n,p,\ratio,\omega(\cdot),\eps) \in (0,1/16)$ such that
\eqn{smally}
$$|(V(Dw_j))_{Q_{\tau}^{A\lambda}(\tilde x,\tilde t)}-(V(Dw_j))_{Q_{\varrho}^{A\lambda}(\tilde x,\tilde t)}|\leq (A\lambda)^{p/2} \ep$$
and
\eqn{smally22}
$$
\left(\mean{Q_{\varrho}^{A\lambda}(\tilde x,\tilde t)} |V(Dw_j)-(V(Dw_j))_{Q_{\varrho}^{A\lambda}(\tilde x,\tilde t)}|^2\, dx \, dt\right)^{1/2} \leq (A\lambda)^{p/2} \ep
$$
hold whenever $(\tilde x,\tilde t) \in \frac 14 Q_j$ and $0< \tau \leq \varrho  \leq R_\ep r_j$; notice that $R_\eps$ is in particular independent of $\lambda$, $A$ and the considered cylinder $Q_j$.
Letting $\tau \to 0$ in \rif{smally} and recalling that $V(Dw_j)$ is continuous yields
\eqn{alpunto}
$$|V(Dw_j(\tilde x,\tilde t))-(V(Dw_j))_{Q_{\varrho}^{A\lambda}(\tilde x,\tilde t)}|\leq (A\lambda)^{p/2} \ep\qquad \forall \ \varrho \in (0,R_\ep r_j]\,.$$
We are now ready to finish the proof with the choice
$$\eps:=\frac{\delta^{p/2}}{c_v 2^{p/2-1}48^N A^{p/2}}, \qquad \qquad \sigma_1:= \frac{A^{(2-p)/2}R_\eps}{32}\,.$$
The constant $c_v$ is the one appearing in \rif{eleV}
Notice that the dependence of $\sigma_1$ upon $n,p,\nu,L,A,\delta, \omega(\cdot)$, as described in the statement, appears through the one implicitly contained in $R_\eps$. Take now $(\tilde y, \tilde s),(\tilde x,\tilde t) \in \sigma_1Q_{j}$; we can assume that $\tilde t \geq \tilde s$ otherwise we can exchange the role between the two points in the next lines. It obviously follows that $$Q_{R_\eps r_j/8}^{A\lambda}(\tilde y, \tilde s) \subset Q_{R_\eps r_j}^{A\lambda}(\tilde x,\tilde t) \subset \frac 14 Q_j\,.$$ Using this last fact, thanks to Jensen's inequality, the one in display \rif{smally22} and using also \rif{mediaprop}, we have
\begin{eqnarray*}
&& |(V(Dw_j))_{Q_{R_{\eps}r_j/8}^{A\lambda}(\tilde y,\tilde s)}-(V(Dw_j))_{Q_{R_{\eps}r_j/8}^{A\lambda}(\tilde x,\tilde t)}|\\
&& \leq \mean{Q_{R_{\eps}r_j/8}^{A\lambda}(\tilde y,\tilde s)} |V(Dw_j)-(V(Dw_j))_{Q_{R_{\eps}r_j/8}^{A\lambda}(\tilde x,\tilde t)}|\, dx \, dt\\
&& \leq 2\mean{Q_{R_{\eps}r_j/8}^{A\lambda}(\tilde y,\tilde s)} |V(Dw_j)-(V(Dw_j))_{Q_{R_{\eps}r_j}^{A\lambda}(\tilde x,\tilde t)}|\, dx \, dt\\
&& \leq 16^N\mean{Q_{R_\eps r_j}^{A\lambda}(\tilde x,\tilde t)} |V(Dw_j)-(V(Dw_j))_{Q_{R_{\eps}r_j}^{A\lambda}(\tilde x,\tilde t)}|\, dx \, dt\\&&
\leq 16^N\left(\mean{Q_{R_\eps r_j}^{A\lambda}(\tilde x,\tilde t)} |V(Dw_j)-(V(Dw_j))_{Q_{R_{\eps}r_j}^{A\lambda}(\tilde x,\tilde t)}|^2\, dx \, dt\right)^{1/2}\\&& \leq 16^N(A\lambda)^{p/2} \ep\,.
\end{eqnarray*}
By using the previous estimate and \rif{alpunto} (actually used also for $(\tilde y,\tilde s)$ instead of $(\tilde x,\tilde t)$) together with triangle inequality we easily gain
\eqn{unaV}
$$
|V(Dw_j(\tilde x,\tilde t))-V(Dw_j(\tilde y,\tilde s))|\leq 48^N (A\lambda)^{p/2} \ep\,.
$$
We are now ready to show \rif{Dw(j) continuous} proving that
\eqn{oscidue}
$$
|Dw_j(\tilde x,\tilde t)-Dw_j(\tilde y,\tilde s)|\leq \delta  \lambda
$$
whenever $(\tilde y, \tilde s),(\tilde x,\tilde t) \in \sigma_1Q_{j}$. First of all, observe that we can assume that either $|Dw_j(\tilde x,\tilde t)|\geq \delta\lambda/2$ or $|Dw_j(\tilde y,\tilde s)|\geq \delta\lambda/2$ holds otherwise we are done. The inequalities in \rif{eleV} and \rif{unaV} then imply \rif{oscidue} as follows:
$$
|Dw_j(\tilde x,\tilde t)-Dw_j(\tilde y,\tilde s)|\leq c_v 48^N(A\lambda)^{p/2} (\delta\lambda/2)^{1-p/2}  \ep = \delta \lambda
$$
and the proof is complete.\end{proof}
\begin{remark}\label{dinim}The proof of the previous result allows in fact, together with the argument given in \cite{KMjmpa}, to get an explicit modulus of continuity for the gradient of solutions of equations with Dini-continuous coefficients. Indeed, the choice of the radius $R_{\eps}r_j$ making \rif{smally} is made to meet a condition of the form
\[
\omega(R_{\eps}r_j) +\int_0^{R_{\eps}r_j} \omega(\varrho) \, \frac{d\varrho}{\varrho} \leq \frac{\eps^{b}}{c}
\]
for some positive constants $b$ and $c$ depending only on $n,p,\nu,L,A$. This gives a modulus of continuity involving a power of the function
$$
r \mapsto \omega(r) +\int_0^{r} \omega(\varrho) \, \frac{d\varrho}{\varrho}
$$
which is in accordance to the known results in the classical elliptic regularity theory.
\end{remark}
We now collect a few results from \cite{KMW, KMpisa} in order to provide oscillation estimates for the functions $v_j$. The next statement is a slight variant of \cite[Theorem 3.1]{KMW}.
\begin{theorem}
\label{sublime}
Let $v_j$ be as in \eqref{CD-local v}. Consider numbers
$$A,B\geq 1 \qquad \mbox{and} \qquad \bar \ep \in (0,1)\,.$$
Then there exists a constant $\sigma_2 \in (0,1/4)$ depending only on $n,p,\nu,L, A,B, \bar \ep$ such that
if
\begin{equation}\label{maina}
\frac{\lambda}{B} \leq \sup_{\sigma_2 Q_{j}} |Dv_j| \leq s+ \sup_{\frac14 Q_j} |Dv_j| \leq A\lambda
\end{equation}
holds, then
\begin{equation}\label{decayexcess0}
\mean{\tau Q_{j}} |Dv_j-(Dv_j)_{\tau Q_{j}}| \, dx \, dt \leq \bar \eps \mean{\tfrac14  Q_{j}} |Dv_j-(Dv_j)_{\tfrac14  Q_{j}}| \, dx \, dt
\end{equation}
holds too, whenever $\tau \in (0,\sigma_2]$.
\end{theorem}
\begin{remark} The essence of the previous result lies in the fact that once the bounds \rif{maina} are satisfied, then solutions to evolutionary $p$-Laplacean type equations satisfy elliptic type decay estimates as in \rif{decayexcess} when framed in the proper intrinsic geometry dictated by \rif{maina}. Indeed, let us denote by $E(f,Q)$ the usual excess functional
\eqn{excess}
$$
E(f,Q):= \mean{Q} |f-(f)_{Q}| \, dx \, dt
$$
which is defined whenever $f$ is an integrable function and $Q$ a measurable set with positive measure; this functional gives an integral measure of the oscillations of $f$ in a subset $Q$. Estimate \rif{decayexcess0} now reads as
\eqn{decayexcess}
$$
E(Dv_j,\tau Q_{j}) \leq  \bar\ep E(Dv_j,\tfrac14 Q_j)\,.$$
Theorem \ref{sublime} gives the natural analog, when passing to the framework of degenerate parabolic equations of $p$-Laplacean type, of the classical results known for solutions to the heat equations. Indeed, Theorem \ref{sublime} holds without assuming \rif{maina} for solutions to \rif{heateq}. This is a classical result of Campanato \cite{C}.
\end{remark}
Using Theorem \ref{sublime} it is possible to give a proof of the H\"older continuity of the gradient of solutions to frozen equations, as for instance shown in \cite[Theorem 3.2]{KMpisa}; see also \cite[Theorem 3.2]{KMW}.
\begin{theorem}
\label{sublimeII}
Let $v_j$ be as in~\eqref{CD-local v}. For every $A\geq 1$ there exist constants $c_4 \equiv c_4(n,p,\nu,L,A)$ and $\alpha \equiv \alpha(n,p,\nu,L,A)$ such that
%\begin{equation}\label{Dv(j) holder}
$$
    \sup_{\frac14 Q_j} |Dv_j| + s \leq A \lambda \qquad \Longrightarrow \qquad \osc_{\tau Q_j} Dv_j  \leq c_4 \tau^\alpha \lambda \qquad \forall \, \tau \in (0,1/4)\,.
    $$
%\end{equation}
\end{theorem}
\begin{remark} Theorems \ref{sublime}-\ref{sublimeII} have been presented in \cite{KMW, KMpisa} actually for solutions to equations of the type
\eqn{senzac}
$$u_t- \divo\, a(Du)=0\,.$$ On the other hand, by following the arguments in \cite{KMW, KMpisa} it is not difficult to see that all the proofs carry out for solutions to equations of the type in display \rif{senzat}. Therefore Theorems \ref{sublime}-\ref{sublimeII} apply to the functions $v_j$ as well, that indeed solve equations as in \rif{senzat}.
\end{remark}
\subsection{Comparison estimates}\label{compi2}
We start this section by a reformulation of a result established in~\cite[Lemma 4.1]{KMW} and \cite[(4.5), (4.6)]{KMW}. We remark that the result there was presented only for equations without coefficients as in \rif{senzac}. %%%ROS omitted the next and for $\mu \in L^1$. 
Nevertheless, the proof works directly for general equations with merely measurable coefficients%%%ROS omitted the nextand genuinely measure data
; the crucial point is the strict monotonicity in the gradient variable.
\begin{lemma}\label{lemma:u-w comparison}
Let $u$  be as in Theorem \ref{main1} and $w_j$ as in \trif{CD-local} with $j \geq 0$.  Let $\tilde \eps \in (0,1/(n+1)]$.
There exist constants $\bar c_1\equiv \bar c_1(n,p,\nu,\tilde \eps)$ and $\bar c_2 \equiv \bar c_2(n,p,\nu)$ such that
\begin{equation}\label{ce}
\left(\mean{Q_j} |Du - D w_j|^q \, dx \, dt\right)^{1/q}
\leq \bar c_1 \lambda \left[\frac1\lambda \frac{|\mu|(Q_j)}{r_j^{N-1}} \right]^{(n+2)/[(p-1)n+p]}
\end{equation}
holds for any $0<q\leq p-1+1/(n+1)-\tilde \eps$. Moreover, the inequalities
\begin{equation}\label{eq:comparison prel 1}
\sup_{t \in T_j} \int_{B_j} |u-w_j| \, dx \leq |\mu|(Q_j)
\end{equation}
and
\begin{equation}\label{eq:comparison prel 2}
\mean{Q_j} \frac{(|Du|+|Dw_j|)^{p-2}|Du-Dw_j|^2}{(\alpha + |u-w_j|)^{\xi}}\,dx \, dt
\leq \bar  c_2 \frac{\alpha^{1-\xi}}{\xi-1} \left[\frac{|\mu|(Q_j)}{\lambda^{2-p}r_j^{N}}\right]
\end{equation}
hold for any $\alpha>0$ and $\xi>1$.
\end{lemma}
We here recall a parabolic Sobolev-Poincar\'e inequality that will be useful in the sequel; we refer to \cite[Chapter 1, Proposition 3.1]{DBbook} for the proof.
\begin{prop}\label{ppoinc} Let $v \in L^{\infty}(T_j ; L^m(B_j))\cap L^{q_2}(T_j ; W^{1,q_2}_0(B_j))$ for $q_2, m \geq 1$. There exists a constant $c$ depending only on $n,q_2, m$ such that the following inequality holds for $q_1= q_2(n+m)/n$:
$$ \int_{Q_j} |v|^{q_1} \, dx \, dt
\leq c \left(\int_{Q_j} |Dv|^{q_2} \, dx \, dt\right)
\left(\sup_{\tau} \int_{B_j} |v (x,\tau)|^m \, dx \right)^{q_2/n}\,.
$$
\end{prop}
Using the previous result and Lemma \ref{lemma:u-w comparison} we get another comparison estimate.
\begin{lemma} \label{cor:comparison}
Let $u$  be as in Theorem \ref{main1} and $w_{j-1},w_j$ as in \trif{CD-local}, with $j \geq 1$.
Then, for any $\tilde \eps \in (0,1/(n+1)]$, there exists a constant $\bar c_3 \equiv \bar c_3(n,p,\nu,\tilde \eps,\sigma)$ such that the inequality
\begin{equation}\label{ce1}
\left(\mean{Q_j} |u-w_j|^{q} \, dx \, dt\right)^{1/q} \leq
\bar c_3  r_{j-1} \lambda \left[ \frac1\lambda \frac{|\mu|(Q_{j-1})}{r_{j-1}^{N-1}} \right]^{(n+p)/[(p-1)n+p]}
\end{equation}
holds whenever $0<q\leq p-1+p/n-\tilde \eps$ and
\begin{equation}\label{ce2}
\left( \mean{Q_j} |Dw_{j-1}-Dw_j|^{q} \, dx \, dt \right)^{1/q} \leq \bar c_3 \lambda \left[ \frac1\lambda \frac{|\mu|(Q_{j-1})}{r_{j-1}^{N-1}} \right]^{(n+2)/[(p-1)n+p]}
\end{equation}
holds whenever $0<q\leq p-1+1/(n+1)-\tilde \eps$.
\end{lemma}
\begin{proof} To prove \rif{ce1}, we use Proposition \ref{ppoinc} with the choice $m=1$,
\[
q_1 := p-1 + \frac{p}{n} -\tilde \eps\,, \qquad q_2 := \frac{n}{n+1} q_1 = p-1 + \frac{1}{n+1} - \frac{n \tilde \eps}{n+1}
\]
and $v \equiv u-w_j \in L^{\infty}(T_j ; L^2(B_j))\cap L^{q_2}(T_j ; W^{1,q_2}_0(B_j))$; recall \rif{eq:comparison prel 1}.
This yields
\begin{eqnarray*}
&& \left(\mean{Q_j} |u-w_j|^{q_1} \, dx \, dt \right)^{1/q_1}
\\ && \qquad
\leq c \bigg(\bigg[\mean{Q_j} |Du-Dw_j|^{q_2} \, dx \, dt\bigg]^{1/q_2}
% \\ && \qquad \qquad \qquad \qquad \times
\bigg[\sup_{\tau} \int_{B_j} |u-w_j | \, dx\bigg]^{1/n} \bigg)^{n/(n+1)}\,.
\end{eqnarray*}
Let $\bar c_1$ be as in Lemma \ref{lemma:u-w comparison} be the constant corresponding to the choice $\tilde \eps n/(n+1)$ instead of $\tilde \eps$.
Substituting~\eqref{ce} and~\eqref{eq:comparison prel 1} into the previous estimate leads to
\begin{equation*}
\begin{split}
& \left(\mean{Q_j} |u-w_j|^{q_1} \, dx \, dt\right)^{1/q_1}
\\  &
\qquad \leq c \bar c_1^{n/(n+1)}  \left(\lambda \left[ \frac1\lambda \frac{|\mu|(Q_{j})}{r_{j}^{N-1}} \right]^{(n+2)/[(p-1)n+p]}
 \left[ |\mu|(Q_{j}) \right]^{1/n}
\right)^{n/(n+1)}
\\  &
\qquad \leq \bar c_3  r_{j-1} \lambda \left[ \frac1\lambda \frac{|\mu|(Q_{j-1})}{r_{j-1}^{N-1}} \right]^{(n+p)/[(p-1)n+p]}\,,
\end{split}
\end{equation*}
which, together with H\"older's inequality, proves \rif{ce1}.
Here we have also used that $N-1=n+1$ and the identity
\[
%\frac{n+2}{(p-1)n+p} + \frac1n = \frac{n+1}{n} \left[ \frac{n+p}{(p-1)n+p} \right]\,.
 \frac{n}{n+1} \left[ \frac{n+2}{(p-1)n+p} + \frac1n \right] =\frac{n+p}{(p-1)n+p}\,.
\]
As for \rif{ce2} we instead argue as follows:
\begin{eqnarray}
\nonumber &&  \left( \mean{Q_j} |Dw_{j-1}-Dw_j|^{q_2} \, dx \, dt \right)^{1/q_2}
\\ \nonumber && \qquad \leq \left(\frac{|Q_{j-1}|}{|Q_j|} \right)^{1/q_2}\left( \mean{Q_{j-1}} |Du-Dw_{j-1}|^{q_2} \, dx \, dt \right)^{1/q_2}
\\ \nonumber && \qquad  \qquad \qquad + \left( \mean{Q_{j}} |Du-Dw_{j}|^{q_2} \, dx \, dt\right)^{1/q_2}
\\ \nonumber && \qquad \leq  \bar c_1 \left[ \left(\frac{|Q_{j-1}|}{|Q_j|} \right)^{1/q_2}  + \frac{r_{j-1}^{N-1}}{r_{j}^{N-1}} \right] \lambda \left[ \frac1\lambda \frac{|\mu|(Q_{j-1})}{r_{j-1}^{N-1}} \right]^{(n+2)/[(p-1)n+p]}
% \\ \nonumber && \qquad  \qquad
% + \bar c_1 \lambda \left[ \frac1\lambda \frac{|\mu|(Q_{j})}{r_{j}^{N-1}} \right]^{(n+2)/[(p-1)n+p]}
 \\ \nonumber  && \qquad \leq \bar c_3 \lambda \left[\frac1\lambda \frac{|\mu|(Q_{j-1})}{r_{j-1}^{N-1}} \right]^{(n+2)/[(p-1)n+p]}\,,
\end{eqnarray}
where we have repeatedly applied~\eqref{ce} and used the fact that we are assuming $p \geq 2$. Now~\eqref{ce2} follows, again by H\"older's inequality as $\tilde \eps$ is arbitrary.
\end{proof}
The following lemma provides one of the key estimates to obtain Theorem~\ref{main1}.
\begin{lemma} \label{lemma:ce3}
Let $u$  be as in Theorem \ref{main1} and $w_{j-1},w_j$ as in \trif{CD-local}, with $j \geq 1$. Suppose further that
\begin{equation}\label{eq:meas cond 1}
    \frac{|\mu|(Q_{j-1})}{r_{j-1}^{N-1}} \leq \lambda
\end{equation}
and that the bounds
\begin{equation}\label{eq:w(j) 1}
 \frac{\lambda}{A} \leq |Dw_{j-1}| \leq A \lambda \quad \textrm{in} \ Q_{j}
 % \,,  \qquad   \frac{\lambda}{A} \leq |Dw_{j}| \leq A \lambda \quad \textrm{in} \ Q_{j+1}
\end{equation}
hold for some $A\geq 1$.
Then there exists a constant $\bar c_4$ depending only on $n,p,\nu,\sigma,A$ such that
\begin{equation}\label{ce3}
\mean{Q_{j}} |Du-Dw_j| \, dx \, dt \leq
\bar c_4 \left[\frac{|\mu|(Q_{j-1})}{r_{j-1}^{N-1}} \right]\,.
\end{equation}
\end{lemma}
\begin{proof}
Let us begin by fixing several parameters appearing in the proof:
%%%TUOMO : I tried to make the preparatory part more clear
\eqn{posi}
$$
\gamma  := \frac{1}{4(p-1)(n+1)}\,, \qquad \xi := 1+2\gamma\,,  \qquad \tilde \eps := \frac1{2(n+1)} \,,
$$
and, throughout the proof, we will apply Lemmas~\ref{lemma:u-w comparison} and~\ref{cor:comparison} with exponents
$$
0 < q \leq \bar q := \xi(p-1) = p-1 + \frac{1}{2(n+1)}\,.
$$
Let $\bar c_1$, $\bar c_2$, $\bar c_3$ be as in Lemma~\ref{lemma:u-w comparison} and Lemma~\ref{cor:comparison}, respectively, corresponding to these choices of $\sigma$ and $\tilde \eps$; therefore they ultimately depend only on $n,p,\nu, \sigma$.  We also set
\eqn{wbar}
$$
\bar w_{j-1} := \frac{w_{j-1}}{\lambda} \, , \qquad \bar w_{j} :=\frac{w_{j}}{\lambda}  \,.
$$
In what follows constants denoted by $c$ will only depend on $n,p,\nu,\sigma,A$ and will in general vary from line to line. We start to estimate the term on the left in~\eqref{ce3} with the aid of~\eqref{eq:w(j) 1} as follows:
\begin{eqnarray}
\nonumber
\mean{Q_j}  |Du-Dw_j| \, dx \, dt &\leq & A^{(p-2)(1+\gamma)} \mean{Q_j} |D\bar w_{j-1} |^{(p-2)(1+\gamma)} |Du-Dw_j| \, dx \, dt
 \\ \nonumber
& \leq& c \mean{Q_j} |D\bar w_{j} - D\bar w_{j-1} |^{(p-2)(1+\gamma)} |Du-Dw_j| \, dx \, dt
 \\ && \label{ce3:1}
 \qquad \qquad +c\mean{Q_j} |D\bar w_{j} |^{(p-2)(1+\gamma)} |Du-Dw_j| \, dx \, dt\,.
\end{eqnarray}
Appealing to H\"older's inequality, together with~\eqref{ce} and~\eqref{ce2}, gives us
\begin{eqnarray}
\nonumber &&
 \mean{Q_j} |D\bar w_{j} - D\bar w_{j-1} |^{(p-2)(1+\gamma)} |Du-Dw_j| \, dx \, dt
 \\ \nonumber &&
 \qquad \leq \lambda^{-(p-2)(1+\gamma)}\left( \mean{Q_j} |D w_{j} - D w_{j-1}  |^{(p-1)(1+\gamma)} \, dx \, dt\right)^{(p-2)/(p-1)}
  \\ \nonumber &&
 \qquad \qquad  \qquad \qquad   \qquad \qquad  \cdot \left( \mean{Q_j} |Du-Dw_j|^{p-1} \, dx \, dt \right)^{1/(p-1)}
 \\ \nonumber &&
 \qquad \leq  c(\sigma)\bar c_3^{(1+\gamma)(p-2)}\bar c_1\lambda \left[\frac1\lambda \frac{|\mu|(Q_{j-1})}{r_{j-1}^{N-1}} \right]^{\left[(1+\gamma)(p-2)+1 \right](n+2)/[(p-1)n+p]}\,.
\end{eqnarray}
But now, as
\eqn{precise}
$$
\frac{\left[(1+\gamma)(p-2)+1 \right](n+2)}{(p-1)n+p} >\frac{(p-1)(n+2)}{(p-1)n+p} \geq 1
$$
precisely for $p\geq 2$, \rif{eq:meas cond 1} implies
\[
\mean{Q_j} |D\bar w_{j} - D\bar w_{j-1} |^{(p-2)(1+\gamma)} |Du-Dw_j| \, dx \, dt \leq c \frac{|\mu|(Q_{j-1})}{r_{j-1}^{N-1}}
\]
with $c \equiv c (n,p,\nu,\sigma)$.
Therefore~\eqref{ce3:1} gives us
\begin{eqnarray}
\nonumber &&
\mean{Q_j}  |Du-Dw_j| \, dx \, dt
 \\ && \label{ce3:2}
 \qquad \leq c\mean{Q_j} |D\bar w_{j} |^{(p-2)(1+\gamma)} |Du-Dw_j| \, dx \, dt + c \frac{|\mu|(Q_{j-1})}{r_{j-1}^{N-1}}\,.
\end{eqnarray}
We then continue to estimate the first term on the right in the above display.  Applying H\"older's inequality,
together with~\eqref{eq:comparison prel 2}, and recalling that $\xi = 1+2 \gamma$, we obtain for any $\alpha>0$ that
\begin{eqnarray}
\nonumber && \mean{Q_j} |D\bar w_j|^{(p-2)(1+\gamma)} |Du-Dw_j| \, dx \, dt
 \\ \nonumber
  && \qquad \leq \mean{Q_j}\left[\lambda^{2-p} \frac{(|Du|+|Dw_j|)^{p-2}|Du-Dw_j|^2}{(\alpha+|u-w_j|)^{\xi}}\right]^{1/2}
   \\ \nonumber && \qquad \qquad \qquad \cdot
 \left[|D\bar w_j|^{(1+2\gamma)(p-2)}(\alpha+|u-w_j|)^\xi\right]^{1/2}\, dx \, dt
 \\ \nonumber && \qquad \leq
 \lambda^{(2-p)/2} \left(\mean{Q_j} \frac{(|Du|+|Dw_j|)^{p-2}|Du-Dw_j|^2}{(\alpha+|u-w_j|)^{\xi}}\, dx \, dt\right)^{1/2}
 \\ \nonumber && \qquad \qquad \qquad\cdot
    \left(\mean{Q_j}|D\bar w_j|^{\xi(p-2)}(\alpha+|u-w_j|)^\xi \, dx \, dt \right)^{1/2}
\\ \nonumber && \qquad  \leq
\left(\frac{\bar c_2}{\xi-1} \right)^{1/2}  \alpha^{(1-\xi)/2}\left[ \frac{|\mu|(Q_j)}{r_{j}^{N}}\right]^{1/2}
 \\  \label{eq:boccardo} && \qquad \qquad \qquad\cdot
 \left(\mean{Q_j}|D\bar w_j|^{\xi(p-2)}(\alpha+|u-w_j|)^\xi \, dx \, dt \right)^{1/2}\,.
%\\ \leq &
%c \left(\frac{|\mu|(Q_j)^{(n+2)/(n+1)}}{|Q_j|} \left(\mean{Q_j}|V(Du)-V(Dw)|^{2q/p} \, dx \, dt \right)^{n/(q(n+1))} \right)^{q/p} \, .
\end{eqnarray}
As the choice of $\alpha$ is still in our disposal, we set
\eqn{aalpha}
$$
\alpha := \left( % \left[\mean{Q_j}|D\bar w_j|^{\xi(p-2)}\, dx \, dt\right]^{-1}
\mean{Q_j}|D\bar w_j|^{\xi(p-2)}|u-w_j|^\xi \, dx \, dt \right)^{1/\xi}  + \delta
$$
for some small positive $\delta \in (0,1)$ to get
\begin{eqnarray}
\nonumber && \left(\mean{Q_j}|D\bar w_j|^{\xi(p-2)}(\alpha+|u-w_j|)^\xi \, dx \, dt \right)^{1/2}
\\ \nonumber && \qquad
\leq 2 \alpha^{\xi/2} \left(\mean{Q_j}|D\bar w_j|^{\xi(p-2)} dx \, dt \right)^{1/2} + 2 \alpha^{\xi/2}\,.
\end{eqnarray}
Notice that since $w_j$ belongs to the parabolic Sobolev space $L^p(T_j;W^{1,p}(B_j))$ and $\xi(p-1)<p$ by \rif{posi}, we have that $\alpha$ is finite by H\"older's inequality. The presence in \rif{aalpha} of the parameter $\delta$, which shall be sent to zero at the end of the proof, guarantees that $\alpha$ is positive. In the above display, the integral on the right can be estimated by mean of Lemma \ref{cor:comparison} as
\begin{eqnarray}
\nonumber && \mean{Q_j}|D\bar w_j|^{\xi(p-2)} dx \, dt
\\ \nonumber && \qquad \leq c \mean{Q_j}|D\bar w_{j-1}-D\bar w_j|^{\xi(p-2)} dx \, dt +c \mean{Q_j}|D\bar w_{j-1}|^{\xi(p-2)} dx \, dt
\\ \nonumber && \qquad \leq c \bar c_3^{\xi(p-2)}
\left[\frac1\lambda \frac{|\mu|(Q_{j-1})}{r_{j-1}^{N-1}} \right]^{\xi(p-2)(n+2)/[(p-1)n+p]} + c A^{\xi(p-2)}  \leq c\,,
\end{eqnarray}
owing to~\eqref{eq:meas cond 1} and~\eqref{eq:w(j) 1}, while the last constant $c$ depends only on $n,p,\nu,\sigma,A$. Thus~\eqref{eq:boccardo}, together with the last two displays, yields
$$
\mean{Q_j} |D\bar w_j|^{(p-2)(1+\gamma)} |Du-Dw_j| \, dx \, dt \leq
c \sqrt {\frac{\alpha}{r_j}} \left[\frac{|\mu|(Q_j)}{r_{j}^{N-1}} \right]^{1/2}
$$
while applying Young's inequality together with an obvious estimation in turn gives
\begin{equation}\nonumber \label{eq:comp inter 1}
\mean{Q_j} |D\bar w_j|^{(p-2)(1+\gamma)} |Du-Dw_j| \, dx \, dt \leq
\frac{c}{\beta}\frac{|\mu|(Q_{j-1})}{r_{j-1}^{N-1}} + \frac{\beta \alpha}{r_{j-1}}
\end{equation}
for all $\beta \in (0,1)$, where $c\equiv c(n,p,\nu,\sigma,A)$. Inserting this into~\eqref{ce3:2} leads to
\begin{equation}
 \label{ce3:3}
\mean{Q_j}  |Du-Dw_j| \, dx \, dt
\leq  \frac{c}{\beta}\frac{|\mu|(Q_{j-1})}{r_{j-1}^{N-1}} + \frac{\beta \alpha}{r_{j-1}}
\end{equation}
again for all $\beta \in (0,1)$, where $c\equiv c(n,p,\nu,\sigma,A)$ is in particular independent of $\beta$. We then focus on $\alpha$, which has been defined in \rif{aalpha}, and split as follows:
\begin{eqnarray}
\nonumber \alpha  \leq
c \left(\mean{Q_j}|D\bar w_{j-1}-D \bar w_j|^{\xi(p-2)}|u-w_j|^{\xi} \, dx \, dt \right)^{1/\xi} & &
\\  +
 c\left(\mean{Q_j}|D \bar w_{j-1}|^{\xi(p-2)}|u-w_j|^{\xi} \, dx \, dt \right)^{1/\xi} + \delta   & =: &  I_1 + I_2 + \delta \,.\label{lpha}
\end{eqnarray}
By \rif{posi}, as $\xi(p-1) = \bar q$, we get by~\eqref{ce1}-\eqref{ce2}, together with H\"older's inequality, that
\begin{eqnarray*}
I_1
& \leq &
 c \left(\mean{Q_j}|D\bar w_{j-1}-D \bar w_j|^{\bar q}\,dx \, dt \right)^{(p-2)/\bar q}\left(\mean{Q_j} |u-w_j|^{\bar q} \, dx \, dt \right)^{1/\bar q }
\\  & \leq &
c \bar c_3^{p-1} r_{j-1} \lambda
\left[\frac1\lambda \frac{|\mu|(Q_{j-1})}{r_{j-1}^{N-1}}
\right]^{[(p-2)(n+2) + n+p]/[(p-1)n+p]}\,.
\end{eqnarray*}
Since
\[
\frac{(p-2)(n+2) + n+p }{(p-1)n+p} = 1 + \frac{2(p-2)}{(p-1)n+p} \geq 1\,,
\]
precisely for $p \geq 2$ we have, in view of~\eqref{eq:meas cond 1}, that
%\begin{equation}\label{eq:I(2)}
$$
I_1 \leq c r_{j-1}
\left[\frac{|\mu|(Q_{j-1})}{r_{j-1}^{N-1}}
\right]
$$%\end{equation}
for $c \equiv c (n,p,\nu,\sigma)$.
On the other hand, using condition~\eqref{eq:w(j) 1} we obtain
%\begin{equation}\label{eq:I(1) 2}
$$
    I_2 \leq c A^{p-2} \left(\mean{Q_j} |u-w_j|^{\xi} \, dx \, dt \right)^{1/\xi}\,.
$$%\end{equation}
Using Proposition \ref{ppoinc}, estimate \eqref{eq:comparison prel 1}, and Young's inequality we get
\begin{eqnarray*}
&& \nonumber \left(\mean{Q_j} |u-w_j|^{\xi} \, dx \, dt \right)^{1/\xi} \leq \left(\mean{Q_j} |u-w_j|^{(n+1)/n} \, dx \, dt \right)^{n/(n+1)}
\\ && \nonumber \qquad \leq c \left(\mean{Q_j} |Du-Dw_j| \, dx \, dt
\left[ \sup_{t \in T_j} \int_{B_j} |u-w_j | \, dx\right]^{1/n} \right)^{n/(n+1)}
\\ && \nonumber \qquad \leq
c \left(\mean{Q_j} |Du-Dw_j| \, dx \, dt \right)^{n/(n+1)} r_{j-1}  \left[\frac{|\mu|(Q_{j-1})}{r_{j-1}^{N-1}} \right]^{1/(n+1)}
\\ &&  \qquad \leq cr_{j-1} \mean{Q_j} |Du-Dw_j| \, dx \, dt  + c r_{j-1} \left[\frac{|\mu|(Q_{j-1})}{r_{j-1}^{N-1}}\right]\,.
%\label{eq:I(1) intermediate}
\end{eqnarray*}
Combining the estimates contained in the last three displays with \rif{lpha} leads to
\[
\alpha \leq c_{*}r_{j-1} \mean{Q_j} |Du-Dw_j| \, dx \, dt  + c_{*} r_{j-1} \left[\frac{|\mu|(Q_{j-1})}{r_{j-1}^{N-1}}\right] + \delta
\]
with $c_{*}\equiv c_{*}(n,p,\nu,\sigma,A)$.
Inserting this finally into~\eqref{ce3:3} with $\beta = 1/(2c_*)$ and then reabsorbing terms and sending $\delta$ to zero finishes the proof.
\end{proof}
Next, a comparison estimate between $w_j$ and $v_j$.
\begin{lemma} \label{lemma:w-v}
Let $w_j$ and $v_j$ be as in~\eqref{CD-local} and~\eqref{CD-local v}, respectively, with $j \geq 0$.
For every $A \geq 1$ there exists a constant $\bar c_5 \equiv \bar c_5(n,p,\nu,L,A)$ such that the following holds:
\begin{eqnarray}
    \nonumber && \sup_{\frac12 Q_j}|Dw_j|+s  \leq A \lambda
    \\  && \nonumber \qquad \Longrightarrow \quad
    \mean{\frac12 Q_j} (|Dw_j| + |Dv_j|)^{p-2} |Dw_j-Dv_j|^2 \, dx \, dt \\ &&\hspace{4cm} + \mean{\frac12 Q_j} |Dw_j-Dv_j|^p \, dx \, dt \leq \bar c_5 \left[\omega(r_j)\right]^{2} \lambda^p\,. \label{w(j) vs v(j) ce2}
\end{eqnarray}
\end{lemma}
\begin{proof}
The result is based on the following estimate:
$$ \mean{\frac12 Q_j} (|Dw_j| + |Dv_j|)^{p-2} |Dw_j-Dv_j|^2 \, dx \, dt
\leq c \left[\omega(r_j)\right]^2 \mean{\frac12 Q_j}(|Dw_j| + s)^p \, dx \,dt $$
that has been proved in~\cite[Lemma 4.3]{KMjmpa}  and \cite{AM}, with a constant $c\equiv c(n,p,\nu,L)$.
The inequality in display \rif{w(j) vs v(j) ce2} follows by trivially estimating
$$
\left[\omega(r_j)\right]^2 \mean{\frac12 Q_j}(|Dw_j| + s)^p \, dx \,dt \leq  \left[\omega(r_j)\right]^2 \sup_{\frac12 Q_j}(|Dw_j| + s)^p\leq \bar c_5 \left[\omega(r_j)\right]^{2} \lambda^p
$$
and taking into account that $p \geq 2$.
\end{proof}
 Finally, along the lines of Lemma \ref{lemma:ce3}, we yet prove another comparison estimate, this time between $u$ and $v_j$.
\begin{lemma} \label{Du vs Dv(j)}
Let $u$  be as in Theorem \ref{main1} and let $w_j$ and $v_j$ be as in~\eqref{CD-local} and~\eqref{CD-local v}, respectively, with $j \geq 1$. Suppose further that \trif{eq:meas cond 1} holds together with
\eqn{eq:w(j) 2}
$$
\left\{
\begin{array}{c}
\displaystyle \sup_{\frac12 Q_j} |Dw_{j}| +s \leq A\lambda\\ [13 pt ]  \displaystyle  \frac{\lambda}{A} \leq |Dw_{j-1}| \leq A \lambda \quad \textrm{in} \ Q_{j}\,,
 % \,,  \qquad   \frac{\lambda}{A} \leq |Dw_{j}| \leq A \lambda \quad \textrm{in} \ Q_{j+1}
\end{array}\right.$$
for some $A\geq 1$. There exist positive constants $\bar c_6 \equiv \bar c_6(n,p,\nu,L,A)$ and $\bar c_7 \equiv \bar c_7(n,p,\nu,\sigma, A)$ such that the following inequality holds:
\begin{equation}\label{Dw(j)-Dv(j)}
\mean{\frac12 Q_j} |Du-Dv_j| \, dx \, dt \leq  \bar c_6 \omega(r_j) \lambda + \bar c_7 \left[\frac{|\mu|(Q_{j-1})}{r_{j-1}^{N-1}} \right]\,.
\end{equation}
\end{lemma}
\begin{proof} We shall keep the notation introduced for Lemma \ref{lemma:ce3}; in particular, we recall \rif{wbar}.
Inequality~\eqref{eq:w(j) 2}$_1$ allows to use~\eqref{w(j) vs v(j) ce2} so that
\begin{eqnarray}\nonumber
&& \mean{\frac12 Q_j} (|Dw_j| + |Dv_j|)^{p-2} |Dw_j-Dv_j|^2 \, dx \, dt \\ && \hspace{3cm} + \mean{\frac12 Q_j} |Dw_j-Dv_j|^p \, dx \, dt\leq \bar c_5 \left[\omega(r_j)\right]^{2} \lambda^p\,. \label{w(j) vs v(j) ce4}
\end{eqnarray}
%Moreover, due to the first inequality in~\eqref{eq:w(j) 2}, by~\eqref{Dw(j) bounded} we have that
%\begin{equation}\label{Dw(j) bounded 2}
%\sup_{\frac12 Q_j } |Dw_j|  \leq 2 c_3 \lambda\,.
%\end{equation}
With $p'=p/(p-1)$, by \rif{eq:w(j) 2}$_2$ we continue estimating as
\begin{eqnarray} \nonumber
 \mean{\frac12 Q_j} |Dw_j-Dv_j| \, dx \, dt  & \leq &  A^{(p-2)/p'} \mean{\frac12 Q_j} |D\bar w_{j-1}|^{(p-2)/p'}|Dw_j-Dv_j| \, dx \, dt
\\ \nonumber &\leq & c\mean{\frac12 Q_j} |D\bar w_{j}|^{(p-2)/p'}|Dw_j-Dv_j| \, dx \, dt
\\ & & + c\mean{\frac12 Q_j} |D\bar w_{j-1}-D\bar w_{j}|^{(p-2)/p'}|Dw_j-Dv_j| \, dx \, dt \label{euna}
\end{eqnarray}
with $c$ depending only on $p$ and $A$ and we are using the notation in \rif{wbar}.
As for the first term in the right hand side of \rif{euna}, we notice that since $p \geq 2$ we have $(p-2)/2 \leq (p-2)/p'$ so that~\eqref{eq:w(j) 2}$_1$ allows to estimate
\begin{eqnarray*} \nonumber
&& \mean{\frac12 Q_j} |D\bar w_{j}|^{(p-2)/p'}|Dw_j-Dv_j| \, dx \, dt
\\ \nonumber & & \qquad \leq A^{(p-2)^2/(2p)} \lambda^{(2-p)/2}\mean{\frac12 Q_j} |D w_{j}|^{(p-2)/2}|Dw_j-Dv_j| \, dx \, dt
\end{eqnarray*}
so that, using H\"older's inequality and~\eqref{w(j) vs v(j) ce4} yields
\begin{eqnarray} \nonumber
&& \mean{\frac12 Q_j} |D\bar w_{j}|^{(p-2)/p'}|Dw_j-Dv_j| \, dx \, dt
\\ \nonumber & & \qquad \leq c\lambda^{(2-p)/2} \left(\mean{\frac12 Q_j} (|Dw_j| + |Dv_j|)^{p-2} |Dw_j-Dv_j|^2 \, dx \, dt \right)^{1/2}
\\  & & \qquad \leq c\lambda^{(2-p)/2} \bar c_5^{1/2} \omega(r_j) \lambda^{p/2}  = c \omega(r_j) \lambda \label{edue}
\end{eqnarray}
for $c \equiv c(n,p,\nu,L,A)$. The second term in the right hand side of \rif{euna} is estimated with the aid of H\"older's inequality,~\eqref{ce2},~\eqref{w(j) vs v(j) ce4} and finally Young's inequality, with conjugate exponents $(p/2,p/(p-2))$ when $p>2$; this means
\begin{eqnarray} \nonumber
&& \mean{\frac12 Q_j} |D\bar w_{j-1}-D\bar w_{j}|^{(p-2)/p'}|Dw_j-Dv_j| \, dx \, dt
\\ \nonumber & & \qquad \leq c \left(\mean{Q_j} |D\bar w_{j-1}-D\bar w_{j}|^{p-2}\, dx \, dt\right)^{1/p'}
% \\ && \qquad \quad  \cdot
\left(\mean{\frac12 Q_j} |Dw_j-Dv_j|^p \, dx \, dt\right)^{1/p}
\\ \nonumber & & \qquad \leq \bar c_3^{(p-2)/p'} \bar c_5^{1/p}\left[\frac1\lambda \frac{|\mu|(Q_{j-1})}{r_{j-1}^{N-1}} \right]^{[(p-2)/p](p-1)(n+2)/[n(p-1)+p]}
\left[\omega(r_j)\right]^{2/p} \lambda
\\ & & \qquad \leq
c \left[\frac1\lambda \frac{|\mu|(Q_{j-1})}{r_{j-1}^{N-1}} \right]^{(p-1)(n+2)/[n(p-1)+p]} \lambda +  \omega(r_j) \lambda \nonumber
\end{eqnarray}
for $c\equiv c(n,p,\nu,L,\sigma, A)$.
Thanks to~\eqref{eq:meas cond 1} and \rif{precise}, we get
\begin{equation} \nonumber
\mean{\frac12 Q_j} |D\bar w_{j-1}-D\bar w_{j}|^{(p-2)/p'}|Dw_j-Dv_j| \, dx \, dt \leq
\tilde c \left[\frac{|\mu|(Q_{j-1})}{r_{j-1}^{N-1}} \right]  + \omega(r_j) \lambda
\end{equation}
with $\tilde c\equiv \tilde c(n,p,\nu,L,\sigma, A)$. By using the last inequality together with \rif{euna} and \rif{edue} we conclude with
%\eqn{simple}
$$
\mean{\frac12 Q_j} |Dw_j-Dv_j| \, dx \, dt \leq c \omega(r_j) \lambda + \tilde c \left[\frac{|\mu|(Q_{j-1})}{r_{j-1}^{N-1}} \right]
$$
for $c\equiv c(n,p,\nu,L,A)$ and $\tilde c\equiv \tilde c(n,p,\nu,L,\sigma, A)$. Note next that the assumptions of this lemma fulfill also the  assumptions of Lemma~\ref{lemma:ce3}. Appealing then to~\eqref{ce3} and triangle inequality finishes the proof.
\end{proof}
%\begin{remark} The essence of the previous lemma relies in the possibility of getting a sharp estimate for the integral appearing in \rif{simple}. Indeed, a simple application of Lemma \ref{lemma:w-v} would give
%$$
%\mean{\frac12 Q_j} |Du-Dv_j| \, dx \, dt \leq  \bar c_5^{2/p} [\omega(r_j)]^{2/p} \lambda $$
%which does not exhibit the correct power for $\omega(r_j)$.
%\end{remark}
\subsection{Comparison lemmas for SOLA}\label{SOLA} In this section we show that the basic comparison Lemmas \ref{lemma:u-w comparison} and \ref{cor:comparison} hold for SOLA in a suitably modified way. This fact ultimately allows to prove Theorem \ref{solath} by mean of the same proofs already given for the other theorems when considered for energy solutions. Before starting, since the setting here is the one defined in Theorem \ref{solath}, we assume the existence of functions $u_h$ being local weak solutions to the equations considered in \rif{appr} and such that
\eqn{convergence}
$$
Du_h \in L^{p}, \quad Du_h \to Du \ \  \mbox{in}\ \ L^{p-1},\quad u_h \to u \quad \mbox{and} \quad Du_h \to Du \ \ a.e.
$$
hold. The measures $\mu_h$ weakly* converge to $\mu$ and \rif{convergencemeasures} holds.
Finally, since all the results we are interested in are local in nature, up to considering subsets compactly contained in $\Omega_T$, we assume w.l.o.g.~that the convergences in \rif{convergence} are valid in the whole $\Omega_T$.

Moreover, in this section we keep the notation already introduced in Section \ref{comp maps} about the cylinders $Q_j$, but the functions
$w_j$ will be defined in a different way: what matters here is not that they are solving Cauchy-Dirichlet problems as \rif{CD-local}, but rather that they satisfy comparison estimates as in Lemmas \ref{lemma:u-w comparison} and \ref{cor:comparison}. Indeed, once these lemmas are available for some suitably regular maps $w_j$ then all the subsequent constructions can be replicated verbatim. More precisely, we have the following:
\begin{lemma}\label{lemma:u-w comparisonApp}
Let $u$ be a SOLA to \trif{maineq} as in Theorem \ref{solath} and let $\tilde \eps \in (0,1/(n+1)]$; moreover, let $\{Q_j\}$ be the sequence of cylinders considered in \trif{shri}. Then there exists a sequence of functions $w_j$
such that
\begin{itemize}
%%%TUOMO : C^0 is already local
\item $ w_j \in C^0(T_j; L^{2}_{\loc}(B_j))\cap L^{p}_{\loc}(T_j; W^{1,p}_{\loc}(B_j)) $
% \qquad \mbox{whenever} \ \ \tilde T \Subset T_j$\\ %%%ROS but then \tilde T becomes T_j
\\
\item $
w_j \ \mbox{is a local weak solution to}\  \partial_t w_j-\divo\, a(x,t,Dw_j)=0\ \ \mbox{in $ Q_j$}
$
\\
\item
There exists a constant $\bar c_1\equiv \bar c_1(n,p,\nu,\tilde \eps)$
such that
\begin{equation}\label{ceAp}
\left(\mean{Q_j} |Du - D w_j|^q \, dx \, dt\right)^{1/q}
\leq \bar c_1 \lambda \left[\frac1\lambda \frac{|\mu|(\lfloor Q_j\rfloor_{\rm par})}{r_j^{N-1}} \right]^{(n+2)/[(p-1)n+p]}
\end{equation}
holds for any $0<q\leq p-1+1/(n+1)-\tilde \eps$
\\
\item
There exists a constant $\bar c_2 \equiv \bar c_2(n,p,\nu)$
\begin{equation}\label{ceApp}
\mean{Q_j} \frac{(|Du|+|Dw_j|)^{p-2}|Du-Dw_j|^2}{(\alpha + |u-w_j|)^{\xi}}\,dx \, dt
\leq \bar  c_2 \frac{\alpha^{1-\xi}}{\xi-1} \frac{|\mu|(\lfloor Q_j\rfloor_{\rm par})}{\lambda^{2-p}r_j^{N}}
\end{equation}
holds for any $\alpha>0$ and $\xi>1$.
\item There exists a constant $\bar c_3 \equiv \bar c_3(n,p,\nu,\tilde \eps,\sigma)$ such that the inequality
\begin{equation}\label{ce1App}
\left(\mean{Q_j} |u-w_j|^{q} \, dx \, dt\right)^{1/q} \leq
\bar c_3  r_{j-1} \lambda \left[ \frac1\lambda \frac{|\mu|(\lfloor Q_{j-1}\rfloor_{\rm par})}{r_{j-1}^{N-1}} \right]^{(n+p)/[(p-1)n+p]}
\end{equation}
holds whenever $0<q\leq p-1+p/n-\tilde \eps$ and
\begin{equation}\label{ce2App}
\left( \mean{Q_j} |Dw_{j-1}-Dw_j|^{q} \, dx \, dt \right)^{1/q} \leq \bar c_3 \lambda \left[ \frac1\lambda \frac{|\mu|(\lfloor Q_{j-1}\rfloor_{\rm par})}{r_{j-1}^{N-1}} \right]^{(n+2)/[(p-1)n+p]}
\end{equation}
holds whenever $0<q\leq p-1+1/(n+1)-\tilde \eps$.
\end{itemize}
\end{lemma}
\begin{proof} The proof goes via approximation; we prove \rif{ceAp}-\rif{ceApp}, the proof of the remaining inequalities being similar. Let us fix a cylinder $Q_j$ and consider the sequence appearing in \rif{convergence} and define
$$
   \tilde w_h \in C^0(T_j ; L^2(B_j))\cap L^p(T_j ; W^{1,p}(B_j))
$$
be the unique solution to the Cauchy-Dirichlet problem
$$
    \begin{cases}
    \partial_t \tilde w_h-\divo\, a(x,t,D\tilde w_h)=0&\mbox{in $ Q_j$}
    \\
    \hfill \tilde w_h=u_h& \mbox{on $\partial_{\rm par} Q_j$\,.}
    \end{cases}
$$
Applying Lemma \ref{lemma:u-w comparison} to this context gives
\begin{equation}\label{ce0}
\left(\mean{ Q_j} |Du_h - D \tilde w_h|^q \, dx \, dt\right)^{1/q}
\leq \bar c_1 \lambda \left[\frac1\lambda \frac{|\mu_h|(Q_j)}{r_j^{N-1}} \right]^{(n+2)/[(p-1)n+p]}
\end{equation}
for any $0<q\leq p-1+1/(n+1)-\tilde \eps$ and moreover
\begin{equation}\label{ce00}
\mean{Q_j} \frac{(|Du_h|+|D\tilde w_h|)^{p-2}|Du_h-D\tilde w_h|^2}{(\alpha + |u_h-\tilde w_h|)^{\xi}}\,dx \, dt
\leq \bar  c_2 \frac{\alpha^{1-\xi}}{\xi-1} \left[\frac{|\mu_h|(Q_j)}{\lambda^{2-p}r_j^{N}}\right]\,.
\end{equation}
From \rif{ce0} and \rif{convergence} it follows that the sequence $\{D\tilde w_h\}$ is bounded in $L^{p-1}(Q_j)$. We now notice that $\tilde w_h$ is an energy solution and solves an equation with Dini-continuous coefficients (see Theorem \ref{thm:Dw(j) continuous} and subsequent Remark \ref{dinim}), moreover all these equations satisfy assumptions \rif{asp}-\rif{aspd} uniformly in $h$. Hence, by interior regularity theory \cite{KM} and in particular by Corollary \ref{cor:Dw(j) bounded 2}, it follows that the the maps $\{D\tilde w_h\}$ are locally uniformly bounded in $L^{\infty}$.  Again by interior regularity theory (see for instance the results in \cite{KMjmpa}) we have that the maps $\tilde w_h$ and $\{D\tilde w_h\}$ are locally uniformly equicontinuous in $Q_j$. Therefore, by Ascoli-Arzel\`a's theorem and a standard diagonal argument, we may assume that, up to a not relabeled sequence, there exists a limit map $w  \in L^{p}_{\loc}(T_j; W^{1,p}_{\loc}(B_j))$ such that $\tilde w_h\to w$ in $L^{p}_{\loc}(T_j; W^{1,p}_{\loc}(B_j))$, $D\tilde w_h \to Dw$ and $\tilde w_h\to w$ locally uniformly and almost everywhere. As a consequence, $w$ weakly solves $\partial_t w-\divo\, a(x,t,Dw)=0$ in $Q_j$. At this stage the proof of \rif{ceAp} and \rif{ceApp} follows taking $w_j :=w$, letting $h \to \infty$ in \rif{ce0}-\rif{ce00} and using Fatou's lemma to deal with the left hand sides and \rif{convergencemeasures} to deal with the right hand sides. Inequalities in displays \rif{ce1App}-\rif{ce2App} can be similarly proved by approximation starting by the analogs of \rif{ce1}-\rif{ce2}, respectively, when written for $u_h$ and $\tilde w_h$ as already done for \rif{ce0}-\rif{ce00}.
\end{proof}
\section{Proof of Theorems~\ref{main1} and \ref{main3}}\label{sec3}
\subsection{Proof of Theorem~\ref{main1}} The proof goes in several steps and involves a rather delicate induction argument. In the following we select a  Lebesgue point of the spatial gradient $(x_0,t_0) \in \Omega_T$, i.e., 
\eqn{lp0}
$$
\lim_{\varrho \to 0}\mean{Q_\varrho(x_0,t_0)} Du \, dx \, dt = Du(x_0,t_0)\,.
$$
Almost every point in $\Omega_T$, with respect to the Lebesgue measure in $\er^{n+1}$, satisfies such a property (see \cite[Chapter 1, Page 8]{Steinbig}).

\subsection*{Step 1: Choice of constants and basic setup.} With $(x_0,t_0)$ being fixed at the beginning we shall verify \rif{unoimpli} with $Q_{2r}^{\lambda}(x_0,t_0)$ instead of $Q_{r}^{\lambda}(x_0,t_0)$; this is of course causes no loss of generality. When finding constants $c, R_0$ such that \rif{unoimpli} works, we choose positive numbers $H_1,H_2$ appearing in the lower bound for $\lambda$:
\begin{equation} \label{lambda1}
    \lambda > H_1 \left(\mean{Q_{2r}^\lambda(x_0,t_0)} (|Du|+s)^{p-1} \, dx \,dt\right)^{1/(p-1)} + H_2 \int_{0}^{2r} \frac{|\mu|(Q_\varrho^\lambda(x_0,t_0))}{\varrho^{N-1}} \frac{d\varrho}{\varrho}\,,
\end{equation}
where $r \in (0,R_0/2)$ and $R_0$ is suitably small, to be determined in due course of the proof. The statement of Theorem \ref{main1} will be then proved with $c := 2\max\{H_1,H_2\}$.  
%%%TUOMO : %%%ROS no any
Without loss of generality, we may assume that $\lambda$ is finite, since otherwise there is nothing to prove. 
In the rest of the proof certain constants will be {\em deliberately} chosen smaller/larger than necessary to emphasize the fact that their role is ``to be very small/large". To begin with, having in mind to apply Theorems \ref{thm:Dw(j) continuous} and \ref{sublime}, we set
\begin{equation}\label{A,B choices}
    A := 1000^{4pN} \max\{c_2,c_3,200\}\,, \quad B := 10^5\,, \quad \delta := 10^{-5} \,, \quad \bar \eps := 4^{-(N+4)}\,,
\end{equation}
where $c_2\equiv c_2(n,p,\nu,L)$ and $c_3\equiv c_3(n,p,\nu,L)$ are as in Corollary~\ref{cor:Dw(j) bounded 2} and Theorem~\ref{thm:Dv(j) bounded}, respectively. In particular, $A$ depends only on $n,p,\nu,L$. Let $\sigma_1 \equiv \sigma_1(n,p,\nu,L, \omega(\cdot)) $ and $\sigma_2\equiv \sigma_2(n,p,\nu,L)$ be as in Theorems~\ref{thm:Dw(j) continuous} and \ref{sublime}, respectively, both corresponding to the choices in \rif{A,B choices}.
Set
\begin{equation}\label{sigma}
  \sigma := \min\{\sigma_1,\sigma_2,1000^{-p/N} (p-1)^{-1/N} A^{-(p-2)/N} , 16^{-Np}\} \in (0,1/4)\,.
\end{equation}
The choices made above guarantee that all fixed parameters $A,B,\bar \eps,\delta,\sigma$ depend only on $n,p,\nu,L,\omega(\cdot)$. Furthermore, let $c_4$ and $\alpha$ be as in Theorem~\ref{sublimeII}, corresponding to the choice of $A$ in \rif{A,B choices}. In this way they both depend only on $n,p,\nu,L$. Next, let $k$ be the smallest integer satisfying
\begin{equation}\label{k}
c_4 \sigma^{k \alpha} \leq \frac{\sigma^{N}}{10^6} \quad \mbox{and} \quad k \geq 2\,.
\end{equation}
As all of $c_4$, $\alpha$ and $\sigma$ depend only on $n,p,\nu,L,\omega(\cdot)$, so does $k\equiv k(n,p,\ratio, \omega(\cdot))$. We now proceed with the choice of $H_1, H_2$ and $R_0$. We set
\begin{equation}\label{H1}
H_1 := 100^{N/(p-1)} 10^6 \sigma^{-4N}\,.% actually \sigma^{-3N} suffices here
\end{equation}
Then, taking $\bar c_1\equiv \bar c_1(n,p,\nu,1/(n+1)) \equiv \bar c_1(n,p,\nu)$ as in Lemma~\ref{lemma:u-w comparison},
$\bar c_3\equiv \bar c_3(n,p,\nu,L,1/(n-1),\sigma) \equiv \bar c_3(n,p,\nu,L, \omega(\cdot))$ from Lemma~\ref{cor:comparison}, and $\bar c_7\equiv \bar c_7(n,p,\nu,L,\linebreak \omega(\cdot), \sigma)\equiv \bar c_7(n,p,\nu,L,\omega(\cdot))$ from Lemma~\ref{Du vs Dv(j)} with $A,\sigma$ as in \rif{A,B choices}-\rif{sigma}, we set
\begin{equation}\label{H2}
H_2 := \left(2^N \sigma^{-N(k+5)}10^6  A \max\{\bar c_1,\bar c_3,\bar c_7\} \right)^{[n(p-1)+p]/(n+2)} \,.
\end{equation}
Notice that since all the constants $H_1$ and $H_2$ are built on ultimately depend on $n,p,\ratio, \omega(\cdot)$, we also have $H_1, H_2 \equiv H_1, H_2(n,p,\ratio, \omega(\cdot))$. We then pass to the choice of $R_0$. Looking then at Lemmas~\ref{lemma:w-v} and \ref{Du vs Dv(j)} and taking $\bar c_5,\bar c_6$ corresponding to the choices of $A$ and $\sigma$ just made in \rif{A,B choices}-\rif{sigma}, we take $R_2\equiv R_2(n,p,\nu,L,\omega(\cdot))$ to be the largest positive number such that
\eqn{asfor}
$$
\bar c_5^{1/p} \sigma^{-N(k+6)}[\omega(R_2)]^{2/p} +  \bar c_6 \sigma^{-(2N+1)}\int_{0}^{R_2} \omega(\varrho) \, \frac{d\varrho}{\varrho} \leq \frac{1}{2^N 10^6}
$$
is satisfied. Finally, with $R_1\equiv R_1(n,p,\nu,L,\omega(\cdot))$ being as in Theorem \ref{thm:Dw(j) bounded 1} and Corollary \ref{cor:Dw(j) bounded 2}, we let
\begin{equation}\label{r4}
  R_0 := \min\{R_1,R_2\}/4\,.
\end{equation}
From now on, with $\lambda$ as in \rif{lambda1}, $2r \leq R_0$ and $Q_{2r}^\lambda(x_0,t_0) \subset \Omega_T$, we adopt and fix the basic set-up described at the beginning of Section \ref{comp maps}. We therefore denote $r_{j}=\sigma^jr$ for $j \geq 0$, $Q_j\equiv Q_{r_j}^{\lambda}(x_0,t_0)$ and the comparison maps $w_j$ and $v_j$ defined in \rif{CD-local} and \rif{CD-local v}, respectively.

We finally record a few immediate consequences of the choices done so far for $H_1, H_2, R_0$. Observe that in \rif{lambda1} the latter potential term may be further estimated from below as follows (recall that $r_0=r$):
\begin{eqnarray}
&&\int_0^{2r} \frac{|\mu|(Q_\varrho^\lambda(x_0,t_0))}{ \varrho^{N-1}} \, \frac{d\varrho}{\varrho}\nonumber \\
 \nonumber && \qquad  =
\sum_{i=0}^\infty \int_{r_{i+1}}^{r_i} \frac{|\mu|(Q_\varrho^\lambda(x_0,t_0))}{ \varrho^{N-1}}  \, \frac{d\varrho}{\varrho}
+ \int_{r_0}^{2r_0} \frac{|\mu|(Q_\varrho^\lambda(x_0,t_0))}{ \varrho^{N-1}}\, \frac{d\varrho} {\varrho}
\\ \nonumber && \qquad \geq
\sum_{i=0}^\infty \frac{|\mu|(Q_{i+1})}{r_i^{N-1}}  \int_{r_{i+1}}^{r_i}  \frac{d\varrho}{\varrho}
+  \frac{|\mu|(Q_{0})}{(2r_0)^{N-1}}   \int_{r_0}^{2r_0}  \frac{d\varrho} {\varrho}
\\ && \qquad =
\sigma^{N-1} \log \left(\frac{1}{\sigma}\right)\sum_{i=0}^\infty \frac{|\mu|(Q_{i+1})}{r_{i+1}^{N-1}}
% \\ \nonumber && \qquad \qquad \qquad +
+
   \frac{ \log 2}{2^{N-1}}\left[\frac{|\mu|(Q_{0})}{r_0^{N-1}} \right]
\geq  \sigma^{N} \sum_{i=0}^\infty \frac{|\mu|(Q_{i})}{ r_{i}^{N-1}} \label{compute}
\end{eqnarray}
so that we conclude with
\begin{equation}\label{lambda2}
    \lambda > H_1 2^{-N/(p-1)} \left(\mean{Q_{0}} (|Du|+s)^{p-1} \, dx \,dt\right)^{1/(p-1)} + H_2 \sigma^{N} \sum_{i=0}^\infty \frac{|\mu|(Q_{i})}{ r_{i}^{N-1}}\,.
\end{equation}
The choice of $H_1$ in \rif{H1} guarantees for instance that
\begin{equation}\label{s}
s \leq \frac{\lambda}{600}\,.
\end{equation}
Notice that in the decomposition \rif{compute} we can always assume that
\eqn{countably}
$$|\mu|(\partial_{\rm par} Q_i)=0$$ for every $i \geq 0$, since this amounts to change the integrand of the first integral in \rif{compute} in countably many points. The choice of $H_2$ in \rif{H2} together with \rif{countably}, allows to deduce that
\begin{eqnarray}\label{trivial}
\nonumber \sum_{i=0}^\infty \frac{|\mu|(Q_{i})}{ r_{i}^{N-1}} &= & \sum_{i=0}^\infty \frac{|\mu|(\lfloor Q_i\rfloor_{\rm par})}{ r_{i}^{N-1}} \\ &\leq & \left(\frac{\sigma^{N(k+4)}}{ 2^N 10^6  A \max\{\bar c_1,\bar c_3,\bar c_7\}}\right)^{[n(p-1)+p]/(n+2)}\lambda \nonumber \\ & \leq& \frac{\sigma^{N(k+4)}}{ 2^N 10^6  A \max\{\bar c_1,\bar c_3,\bar c_7\}}\lambda   \leq  \lambda \,.
\end{eqnarray}
Notice that in the second estimate above we have used the fact that, since $p \geq 2$, then $n(p-1)+p \geq n+2$, and that the quantity in brackets is smaller than $1$.

Recall now the choice of $R_0$ in \rif{r4} and observe that, if $r\equiv r_0 \in (0, R_2/4]$, then
\begin{eqnarray} \nonumber
\int_{0}^{R_2} \omega(\varrho) \, \frac{d\varrho}{\varrho} & =  & \sum_{i=0}^{\infty} \int_{r_{i+1}}^{r_i} \omega(\varrho) \, \frac{d\varrho}{\varrho} + \int_{r_0}^{R_2} \omega(\varrho) \, \frac{d\varrho}{\varrho}\\ & \geq &\log \left(\frac{1}{\sigma}\right) \sum_{i=0}^\infty \omega(r_{i+1}) + \log 4 \omega(r_0) \geq \sigma
\sum_{i=0}^\infty \omega(r_i) \label{compute0}
\end{eqnarray}
holds, so that \rif{asfor} implies
\begin{equation}
\label{omega sum}
\bar c_5^{1/p} \sigma^{-N(k+4)}[\omega(r_0)]^{2/p} + \bar c_6 \sum_{i=0}^\infty \omega(r_i) \leq \frac{\sigma^{2N}}{2^N 10^6}\,.
\end{equation}
\begin{remark}When dealing with equations of the type in \rif{senzat}, i.e., equations without dependence on $x$, Theorem \ref{thm:Dw(j) continuous} is not any longer needed and therefore no dependence on $\omega(\cdot)$ occurs in the constants. Actually, we do not need the comparison functions $w_j$ and we can just use $v_j$. \end{remark}
\begin{remark}An ambiguity occurs when considering \rif{countably}, since \rif{countably} is automatically satisfied when $\mu \in L^1$, which is precisely the case when proving Theorem \ref{main1}. Here we nevertheless wanted to emphasize the fact that we can assume \rif{countably} at this stage also when $\mu$ is genuinely a measure. This fact will play a role when proving Theorem \ref{solath} below. \end{remark}
\subsection*{Step 2: Exit time argument.}
After having fixed the relevant constants in the previous step, we consider, for indexes $i \geq 1$, the quantities
%\begin{equation}\label{Ci}
$$
    C_i := \sum_{m=-1}^{0} \left(\mean{Q_{i+m}} |Du|^{p-1} \, dx \, dt \right)^{1/(p-1)} + 2\sigma^{-N} \mean{Q_{i}} |Du-(Du)_{Q_i}| \, dx \, dt \,.
$$%\end{equation}
By~\eqref{H1} and \eqref{lambda2} it follows that $$C_1\leq 6 \sigma^{-N-\frac{N}{p-1}}\left(\mean{Q_{0}} |Du|^{p-1} \, dx \, dt \right)^{1/(p-1)} \leq \frac{\lambda}{1000}\,.$$ Furthermore, without loss of generality, we may assume that there is an exit time index $i_e\geq 1$ such that $C_i >\lambda/1000$ whenever $i>i_e$ and $C_{i_e} \leq \lambda/1000$. Indeed, if such an index does not exist, we have -- in view of $C_1 \leq \lambda/1000$ -- a subsequence $(i_j)_j$ of indexes such that $C_{i_j} \leq \lambda/1000$ as $j\to \infty$. But, as we assume that $(x_0,t_0)$ is a Lebesgue point of $Du$, then also
\eqn{lpo}
$$
|Du(x_0,t_0)| = \lim_{j\to \infty} |(Du)_{Q_{i_j}}| \leq \limsup_{j\to \infty}\, C_{i_j} \leq \frac{\lambda}{1000}
$$
would hold and the proof would be complete. Thus, from now on, we shall work under the assumption
\begin{equation}\label{exit time}
   C_i > \frac{\lambda}{1000} \quad \textrm{for} \ \ i \in \{ i_e+1,i_e+2,\ldots\}\,, \quad \textrm{and} \quad C_{i_e} \leq \frac{\lambda}{1000}\,.
\end{equation}
Note that in \rif{lpo} we are using a limit computed on a sequence of shrinking intrinsic cylinders, therefore different from those considered in \rif{lp0}. On the other hand, define the set $\mathcal{L}_\lambda$ as
%\eqn{lebe}
$$
\mathcal{L}_\lambda := \left\{(x_0,t_0) \in \Omega_T\ : \ \lim_{\varrho \to 0}\mean{Q_\varrho^\lambda(x_0,t_0)} Du \, dx \, dt = Du(x_0,t_0)\right\}
$$
for $\lambda > 0$. Basic properties of maximal operators - see for instance \cite[Chapter 1, Page 8]{Steinbig} - imply that this set is actually independent of $\lambda$ and, in particular, $\mathcal{L}_\lambda = \mathcal{L}_1 =: \mathcal{L}$ for all $0<\lambda<\infty$ and therefore \rif{lpo} is completely justified in view of the assumed property in \rif{lp0}.
\subsection*{Step 3: Induction scheme.} In order to prove \rif{unoimpli} using~\eqref{exit time}, we apply an induction argument. To shorten the notation, we set, for $i \geq 0$,
\eqn{derive}
$$
E_i := \mean{Q_{i}} |Du-(Du)_{Q_i}| \, dx \, dt \,, \qquad a_i :=  |(Du)_{Q_i}|\,.
$$
We shall consider, in our iterative setting, the following conditions:
\begin{equation}
%\label{induction Du}
\nonumber
\textrm{Ind}_1(j): \qquad
\left(\mean{Q_{j-1}} |Du|^{p-1} \, dx \, dt\right)^{1/(p-1)}+\left(\mean{Q_{j}} |Du|^{p-1} \, dx \, dt   \right)^{1/(p-1)} \leq \lambda
\end{equation}
%%%TUOMO : every -> given
for given integer $j\geq i_e$ and
\begin{equation}
\nonumber
\textrm{Ind}_2(j): \qquad  \sum_{i=i_e+1}^{j} E_i \leq \frac12 \sum_{i=i_e}^{j-1} E_i + \frac{2\bar c_6}{\sigma^N} \sum_{i=i_e}^{j-1} \omega(r_i)   \,\lambda+ \frac{2\bar c_7}{\sigma^N} \sum_{i=i_e-1}^{j-2} \frac{|\mu|(Q_i)}{r_i^{N-1}}
\end{equation}
for given integer $j > i_e$. The constants $\bar c_6, \bar c_7$ are those defined in Lemma \ref{Du vs Dv(j)} with the choices of $A$ and $\sigma$ made in \rif{A,B choices} and \rif{sigma}, respectively; in this way, they ultimately depend on $n,p,\ratio, \omega(\cdot)$.

The goal is now to show, by induction, that $\textrm{Ind}_1(j)$ holds for every integer $j \geq i_e$ and $\textrm{Ind}_2(j)$ holds for every integer $j > i_e$. Indeed, this will immediately prove Theorem \ref{main1} as
\[
|Du(x_0,t_0)| = \lim_{j\to \infty} a_j \leq \limsup_{j\to \infty} \left(\mean{Q_{j}} |Du|^{p-1} \, dx \, dt \right)^{1/(p-1)} \leq \lambda\,.
\]
Now let us remark that $\textrm{Ind}_1(i_e)$ is automatically satisfied since
\eqn{est}
$$
\sum_{m=-1}^0 \left(\mean{Q_{i_e+m}} |Du|^{p-1} \, dx \, dt \right)^{1/(p-1)}\leq C_{i_e} \leq \frac{\lambda}{1000}\,.
$$
Therefore, the rest of the proof develops according to the following scheme:
\eqn{gra1}
$$
\textrm{Ind}_1(i_e) \quad   \Longrightarrow \quad \textrm{Ind}_2(i_e+1)\,,
$$
\eqn{gra2}
$$\left\{\begin{array}{c}
\textrm{Ind}_1(j)\\ [5 pt]
\textrm{Ind}_2(j)
\end{array}\right. \quad  \Longrightarrow  \quad \textrm{Ind}_2(j+1) \qquad \forall \  j > i_e
$$
and
\eqn{gra3}
$$\left\{
\begin{array}{c}
\textrm{Ind}_1(j)\\ [5 pt]
\textrm{Ind}_2(j+1)
\end{array} \right. \quad   \Longrightarrow \quad \textrm{Ind}_1(j+1) \qquad \forall \ j \geq i_e\,.
$$
We remark that in the following, unless otherwise stated, whenever we are considering $\textrm{Ind}_1(j)$ we will do it in the general case $j \geq i_e$, while, when considering $\textrm{Ind}_2(j)$ we will do it for all indexes $j > i_e$.
%Indeed, set
%$$
%X:= \left\{j > i_e \, : \, \mbox{either}\ \textrm{Ind}_1(j) \ \mbox{or}\ \textrm{Ind}_2(j) \ \mbox{is false}\right\}\,.
%$$
%and let show that $X$ is empty. Notice that we already know, by \rif{gra1}, that both $\textrm{Ind}_1(i_e)$ and $\textrm{Ind}_2(i_e+1)$ are true. Now, assume by contradiction that $X$ is non-empty. It therefore admits a minimum $j_m>i_e$. The first case is when $j_m=i_e+1$; then, since $\textrm{Ind}_2(i_e+1)$ holds, it follows that $\textrm{Ind}_1(i_e+1)$ must be false. But then \rif{gra3} used with $j=i_e$ implies that $\textrm{Ind}_1(i_e+1)$ holds and this is impossible. Therefore it must be $j_m>i_e+1$. In this case observe that both $\textrm{Ind}_1(j_m-1)$ and $\textrm{Ind}_2(j_m-1)$ are true. Using \rif{gra2} with $j=j_m-1$ it follows that $\textrm{Ind}_2(j_m)$ holds; in turn, using \rif{gra3} again for $j=j_m-1$ we have that also $\textrm{Ind}_1(j_m)$ holds and therefore $j_m$ does not belong to $X$, which is a contradiction. Therefore $X$ is empty.

\subsection*{Step 4: Upper bounds implied by $\textrm{Ind}_1(j)$.} Towards the proof of the induction step, here we assume that $\textrm{Ind}_1(j)$ holds for a certain index $j \geq i_e$ and exploit a few consequences of this.
With the integer $k$ being defined as in \rif{k}, observe that Lemma~\ref{lemma:u-w comparison} and~\eqref{trivial} imply that whenever $l\in \{0,1,\ldots,k+1\}$ the following holds:
\begin{eqnarray}
\nonumber && \left(\mean{Q_{j-1+l}} |Du-Dw_{j-1}|^{p-1} \, dx \,dt \right)^{1/(p-1)}
\\ \nonumber  && \qquad \leq \sigma^{-\frac{Nl}{p-1}}\left(\mean{Q_{j-1}} |Du-Dw_{j-1}|^{p-1} \, dx \,dt \right)^{1/(p-1)}
\\ && \qquad \leq \bar c_1\sigma^{-N(k+1)}  \lambda \left[\frac1\lambda \frac{|\mu|(Q_{j-1})}{r_{j-1}^{N-1}} \right]^{(n+2)/[(p-1)n+p]} \leq \frac{ \sigma^N}{2^N 10^6}\lambda  \,.  \label{ind u-w(j-1)}
\end{eqnarray}
Similarly, the inequalities
%\begin{eqnarray}\label{ind u-w(j)again}
\begin{equation}\label{ind u-w(j)again}
\left(\mean{Q_{j+l}} |Du-Dw_{j}|^{p-1} \, dx \,dt \right)^{1/(p-1)} \leq\frac{\sigma^N }{2^N 10^6} \lambda
\end{equation}
%\end{eqnarray}
hold as well for $l\in \{0,1,\ldots,k+1\}$.
Using~\eqref{ind u-w(j-1)} with $l=0$ and $\textrm{Ind}_1(j)$ we get
\begin{eqnarray*}\nonumber
    && \left(\mean{Q_{j-1}} |Dw_{j-1}|^{p-1} \, dx \,dt\right)^{1/(p-1)}
    \\ && \qquad  \qquad  \leq \frac{ \sigma^N}{2^N10^6} \lambda +  \left(\mean{Q_{j-1}} |Du|^{p-1} \, dx \,dt\right)^{1/(p-1)} \leq 2\lambda \,.%\label{propro}\,.
\end{eqnarray*}
By Corollary~\ref{cor:Dw(j) bounded 2} (applied to $w_{j-1}$) we then obtain, also recalling~\eqref{s},  that
\begin{equation}\label{Dw(j-1) bound}
    \sup_{\frac12 Q_{j-1}} |Dw_{j-1}|  + s \leq c_2(\lambda+s) + c_2 \lambda^{2-p} \mean{Q_{j-1}} (|Dw_{j-1}|+s)^{p-1} \, dx \,dt \leq A\lambda\,.
\end{equation}
In a completely similar way, by using this time \rif{ind u-w(j)again} instead of \rif{ind u-w(j-1)}, and again $\textrm{Ind}_1(j)$, we can prove
\eqn{simi}
$$
\left(\mean{Q_{j}} |Dw_{j}|^{p-1} \, dx \,dt\right)^{1/(p-1)} \leq 2\lambda
$$
and then
\begin{equation}\label{ind Dw(j) 1}
 \sup_{\frac12 Q_{j}} |Dw_{j}|+ s \leq A\lambda  \,.
\end{equation}
Furthermore, keeping~\eqref{omega sum} in mind, estimate~\eqref{ind Dw(j) 1} and Lemma~\ref{lemma:w-v} allow us to deduce
\eqn{ind v(j)-w(j)}
$$\left(\mean{\frac12 Q_{j}} |Dw_j-Dv_j|^{p-1} \, dx \,dt \right)^{1/(p-1)}
\leq  \bar c_5^{1/p}\left[\omega(r_{j})\right]^{2/p} \lambda  \leq \frac{\sigma^{N(k+6)}}{2^N10^6} \lambda
$$
so that, for $l \in \{1,\ldots,k+1\}$ it holds that
$$
\left(\mean{Q_{j+l}} |Dw_j-Dv_j|^{p-1} \, dx \,dt \right)^{1/(p-1)} \leq \frac{\sigma^{2N}}{2^N10^6} \lambda\,.
$$
Combining the above estimate with~\eqref{ind u-w(j)again} gives
%\begin{equation}\label{ind u-v(j) 1}
%\left(\mean{\frac12 Q_{j}} |Du-Dv_{j}|^{p-1} \, dx \,dt \right)^{1/(p-1)}   \leq \frac{\sigma^N }{10^5} \lambda
%\end{equation}
%and
\begin{equation}\label{ind u-v(j) 2}
\left(\mean{Q_{j+l}} |Du-Dv_{j}|^{p-1} \, dx \,dt \right)^{1/(p-1)} \leq \frac{\sigma^N }{10^6} \lambda
\end{equation}
%and
%\begin{equation}\label{ind u-v(j) 2meno}
%\left(\mean{Q_{j-1+l}} |Du-Dv_{j-1}|^{p-1} \, dx \,dt \right)^{1/(p-1)} \leq \frac{\sigma^N }{10^5} \lambda
%\end{equation}
for all $l \in \{1,\ldots,k+1\}$. Next, observe that \rif{simi} and \rif{ind v(j)-w(j)} imply
\begin{eqnarray*}\nonumber
    && \left(\mean{\frac12 Q_{j}} |Dv_{j}|^{p-1} \, dx \,dt\right)^{1/(p-1)}
    \\ &&   \quad  \leq \left(\mean{\frac 12 Q_{j}} |Dw_j|^{p-1} \, dx \,dt\right)^{1/(p-1)} +  \left(\mean{\frac 12Q_{j}} |Dw_j-Dv_j|^{p-1} \, dx \,dt\right)^{1/(p-1)} \\ && \quad \leq 2^{N+1}\lambda +  \lambda \leq 4^N \lambda
\end{eqnarray*}
so that Theorem~\ref{thm:Dv(j) bounded}, \rif{A,B choices} and \rif{s} finally yield
\begin{equation}\label{ind Dv(j) bound}
     \sup_{\frac14 Q_j} |Dv_j|+s \leq c_3 (\lambda +s)+ c_3 \lambda^{2-p} \mean{\frac12Q_j} (|Dv_j|+s)^{p-1} \, dx \, dt \leq A\lambda\,.
\end{equation}
\subsection*{Step 5: Lower bounds implied by $\textrm{Ind}_1(j)$.} Here we still exploit a few consequences of assuming $\textrm{Ind}_1(j)$ for some $j \geq i_e$. In particular, we derive suitable lower bounds for $Dw_{j-1}$ and $Dv_{j}$. First, let us first show a few oscillation reduction estimates.
Since we have already established the upper bound for $Dw_{j-1}$ in~\eqref{Dw(j-1) bound}, Theorem~\ref{thm:Dw(j) continuous} applied to $w_{j-1}$ (with $\delta=10^{-5}$ as in \rif{A,B choices} and recalling also the choice made in \rif{sigma} that guarantees $\sigma \leq \sigma_1$ so that $Q_j \subset \sigma_1 Q_{j-1}$) gives
\begin{equation}\label{Dw(j-1) osc}
    \osc_{Q_{j}} Dw_{j-1} \leq \frac{\lambda}{10^5}\,.
\end{equation}
For $Dv_{j}$ we instead have~\eqref{ind Dv(j) bound} and hence Theorem~\ref{sublimeII} and~\eqref{k} imply
%\begin{equation}\label{Dv(j) osc}
$$
    % \osc_{Q_{j}} |Dv_{j-1}| \leq c_4 \sigma^{\alpha} \lambda \leq \lambda/1200\,, \qquad
    \osc_{Q_{j+k}} Dv_{j} \leq c_4 \sigma^{\alpha k} \lambda \leq \frac{\sigma^{N}}{10^6} \lambda\,.
$$%\end{equation}
Using this together with~\eqref{mediaprop} and~\eqref{ind u-v(j) 2} gives
\begin{eqnarray}
&& \nonumber 2\sigma^{-N} \mean{Q_{j+k}} |Du-(Du)_{Q_{j+k}}| \, dx \, dt
\\ && \nonumber\qquad  \leq
4\sigma^{-N} \mean{Q_{j+k}} |Dv_j-(Dv_j)_{Q_{j+k}}| \, dx \, dt+ 4\sigma^{-N}
\mean{Q_{j+k}} |Du-Dv_{j}| \, dx \, dt
\\ && \nonumber\qquad  \leq
4\sigma^{-N} \osc_{Q_{j+k}} Dv_{j} + 4\sigma^{-N}
\mean{Q_{j+k}} |Du-Dv_{j}| \, dx \, dt
  \leq \frac{\lambda}{10^5}\,.
\end{eqnarray}
But, as $C_{j+k}>\lambda/1000$ for $j \geq i_e$, we then necessarily have that
\eqn{sottosotto}
$$
    \sum_{m=-1}^{0} \left(\mean{Q_{j+m+k}} |Du|^{p-1} \, dx \, dt\right)^{1/(p-1)} \geq \frac{\lambda}{2000}\,.
$$
Next, again by using triangle inequality and \rif{ind u-w(j-1)} (for $l=k, k+1$), we have
\begin{eqnarray}
\nonumber
&& \sum_{m=-1}^{0} \left(\mean{Q_{j+m+k}} |Du|^{p-1} \, dx \, dt\right)^{1/(p-1)}  %\\ \nonumber && \qquad \qquad \leq  \sum_{m=-1}^{0}\bar c_1\sigma^{-Nk}  \lambda \left[\frac1\lambda \frac{|\mu|(Q_{j-1})}{r_{j-1}^{N-1}} \right]^{(n+2)/[(p-1)n+p]}\\ %&& \qquad \qquad \qquad \qquad \nonumber + \sum_{m=-1}^{0} \left(\mean{Q_{j+m+k}} |Dw_{j-1}|^{p-1} \, dx \, dt\right)^{1/(p-1)}
\\ \label{tri1} & & \qquad \qquad  \leq   \frac{ \lambda}{10^6}  + \sum_{m=-1}^{0} \left(\mean{Q_{j+m+k}} |Dw_{j-1}|^{p-1} \, dx \, dt\right)^{1/(p-1)}
 \end{eqnarray}
so that, as $k \geq 2$ and \rif{sottosotto} holds, we also get
\begin{eqnarray}
\nonumber 2 \sup_{Q_{j}} |Dw_{j-1}| &\geq & \sum_{m=-1}^{0} \left(\mean{Q_{j+m+k}} |Dw_{j-1}|^{p-1} \, dx \, dt\right)^{1/(p-1)}\\ &\geq & \frac{\lambda}{2000} - \frac{\lambda}{10^6} \geq \frac{3\lambda}{10^4}\,.\label{point}
 \end{eqnarray}
Arguing as for \rif{tri1}, and using~\eqref{ind u-v(j) 2}, we this time have
\begin{eqnarray*}
\nonumber
&& \sum_{m=-1}^{0} \left(\mean{Q_{j+m+k}} |Du|^{p-1} \, dx \, dt\right)^{1/(p-1)}
\\ & & \qquad \qquad  \leq   \frac{ \lambda}{10^6}  + \sum_{m=-1}^{0} \left(\mean{Q_{j+m+k}} |Dv_{j}|^{p-1} \, dx \, dt\right)^{1/(p-1)}
 \end{eqnarray*}
so that, using again \rif{sottosotto}, and recalling that $k \geq 2$, we conclude with
\eqn{point2}
$$
2 \sup_{Q_{j+1}} |Dv_{j}| \geq \sum_{m=-1}^{0}\left(\mean{Q_{j+m+k}} |Dv_j|^{p-1} \, dx \, dt\right)^{1/(p-1)}  \geq \frac{3\lambda}{10^4}\,.
$$
The inequality in display \rif{point} yields the existence of a point $(\tilde x, \tilde t) \in Q_{j}$ such that $|Dw_{j-1}(\tilde x, \tilde t)|\geq \lambda/10^4$ and therefore the oscillation control in~\eqref{Dw(j-1) osc}, recalling also~\eqref{Dw(j-1) bound} and that $Q_j \subset \frac 12 Q_{j-1}$, gives
\begin{equation}\label{Dw(j-1) upper and lower}
\frac{\lambda}{A} \leq \frac{\lambda}{10^5} \leq |Dw_{j-1}| \leq A\lambda \qquad \textrm{in} \ Q_j\,.
\end{equation}
Finally, \rif{point2} and~\eqref{ind Dv(j) bound}, and the fact that $Q_{j+1} \subset \frac14 Q_{j}$, give
\begin{equation}\label{Dv(j) upper and lower}
\frac{\lambda}{B} \equiv \frac{\lambda}{10^5} \leq \sup_{Q_{j+1}} |Dv_{j}| \leq \sup_{\frac14 Q_{j}} |Dv_{j}| + s \leq A\lambda\,.
\end{equation}
\subsection*{Step 6: Further consequences of $\textrm{Ind}_1(j)$.} We again assume $\textrm{Ind}_1(j)$ for an arbitrary chosen $j \geq i_e$. Now, on one hand~\eqref{trivial},~\eqref{ind Dw(j) 1} and \eqref{Dw(j-1) upper and lower}
allow to apply Lemma~\ref{Du vs Dv(j)} and obtain
\begin{equation}\label{decisive comp}
\mean{\frac12 Q_j} |Du-Dv_j| \, dx \, dt \leq  \bar c_6 \omega(r_j) \lambda + \bar c_7 \left[\frac{|\mu|(Q_{j-1})}{r_{j-1}^{N-1}} \right]\,.
\end{equation}
On the other hand, recalling that $\sigma Q_j=Q_{j+1}$, Theorem~\ref{sublime} with $\bar \eps=4^{-(N+4)}$ is at our disposal by~\eqref{Dv(j) upper and lower} and it yields
\eqn{decisive comp2}
$$
\mean{Q_{j+1}} |Dv_j-(Dv_{j})_{Q_{j+1}}| \, dx \, dt \leq 4^{-(N+4)} \mean{\frac14 Q_{j}} |Dv_j-(Dv_{j})_{\frac14 Q_{j}}| \, dx \, dt\,.
$$
Now, by~\eqref{decisive comp}, recalling the definitions in \rif{derive} and using \rif{mediaprop} repeatedly, both
\begin{eqnarray}
\nonumber
&& \mean{\frac14 Q_{j}} |Dv_j-(Dv_{j})_{\frac 14 Q_{j}}| \, dx \, dt
\\ \nonumber &&\qquad \leq  2^{2N+1} \mean{Q_{j}} |Du-(Du)_{Q_{j}}| \, dx \, dt + 2^{N+2}\mean{\frac12 Q_{j}} |Du-Dv_{j}| \, dx \, dt
\\ \nonumber &&\qquad \leq 2^{2N+1} E_j + 2^{N+2} \bar c_6 \omega(r_j) \lambda + 2^{N+2}\bar c_7 \left[\frac{|\mu|(Q_{j-1})}{r_{j-1}^{N-1}} \right]
\end{eqnarray}
and
\begin{eqnarray}
\nonumber
&&  2\mean{Q_{j+1}} |Dv_j-(Dv_{j})_{Q_{j+1}}| \, dx \, dt
\\ \nonumber &&\qquad \geq \mean{Q_{j+1}} |Du-(Du)_{Q_{j+1}}| \, dx \, dt - \sigma^{-N} \mean{\frac 12 Q_{j}} |Du-Dv_{j}| \, dx \, dt
\\ \nonumber &&\qquad \geq E_{j+1} - \sigma^{-N} \bar c_6 \omega(r_j) \lambda -\sigma^{-N} \bar c_7 \left[\frac{|\mu|(Q_{j-1})}{r_{j-1}^{N-1}} \right]
\end{eqnarray}
hold. Combining the inequalities in the last three displays gives
\eqn{summa}
$$
E_{j+1} \leq \frac12 E_j + \frac{2\bar c_6}{\sigma^N} \omega(r_j) \lambda + \frac{2 \bar c_7}{\sigma^N} \left[\frac{|\mu|(Q_{j-1})}{r_{j-1}^{N-1}} \right]\,.
$$

\subsection*{Step 7: Verification of $\textrm{Ind}_2(i_e+1)$ and $\textrm{Ind}_2(j+1)$.} Here we prove \rif{gra1} and \rif{gra2}. The outcome of Step 6 is that \rif{summa} holds whenever $\textrm{Ind}_1(j)$ holds, for every $j \geq i_e$, therefore, since $\textrm{Ind}_1(i_e)$ holds as established in \rif{est}, taking $j= i_e$ in \rif{summa} yields $\textrm{Ind}_2(i_e+
1)$. We now come to the proof of \rif{gra2}, that is the validity of $\textrm{Ind}_2(j+1)$; this simply follows by summing \rif{summa}, that holds  by the assumed validity of $\textrm{Ind}_1(j)$, to the inequality yielded by the definition of $\textrm{Ind}_2(j)$, which is also assumed in \rif{gra2}.

\subsection*{Step 8: Bounds for $a_j$ and $E_j$.} It remains to prove \rif{gra3}. For this we now assume that $\textrm{Ind}_1(j)$ and $\textrm{Ind}_2(j+1)$ hold for some $j \geq i_e$ and derive bounds for the quantities defined in \rif{derive}. By $\textrm{Ind}_2(j+1)$ and easy manipulations, we note that~\eqref{trivial},~\eqref{omega sum},
and $C_{i_e} \leq \lambda/1000$ imply
\begin{eqnarray}\nonumber
    \sum_{i=i_e}^{j+1} E_i &\leq & 2E_{i_e} + \frac{4\bar c_6}{\sigma^N} \sum_{i=0}^{\infty} \omega(r_i) \, \lambda + \frac{4\bar c_7}{\sigma^N} \sum_{i=0}^{\infty} \frac{|\mu|(Q_i)}{r_i^{N-1}} \\
    &\leq & 2E_{i_e} + \frac{ \sigma^{N} \lambda}{1000} \leq \sigma^{N}C_{i_e} + \frac{ \sigma^{N} \lambda}{1000}\leq \frac{ \sigma^{N}\lambda }{500}  \,. \label{E sum}
\end{eqnarray}
%It trivially follows from this that
%\begin{equation}\label{E(j)}
%E_j \leq \lambda \sigma^{N}/1000 \, .% \leq 4^{-(p-1)} A^{-(p-2)} \lambda\,.
%\end{equation}
Using this together with the obvious estimation (valid whenever $i \geq 0$)
\begin{eqnarray*}
a_{i+1} - a_i &\leq & \mean{Q_{i+1}} |Du-(Du)_{Q_i}| \, dx \, dt \\
&\leq &  \frac{|Q_i|}{|Q_{i+1}|}\mean{Q_{i}} |Du-(Du)_{Q_i}| \, dx \, dt
= \sigma^{-N} E_i
\end{eqnarray*}
we get, after telescoping the previous inequalities and using \rif{E sum}, that
\begin{equation}\label{a bound}
    a_{l+1} \leq a_{i_e} + \sigma^{-N} \sum_{i=i_e}^l E_i \leq \frac{\lambda}{1000} + \frac{\lambda}{500}  \leq \frac{\lambda}{200}
\end{equation}
for all $l \in \{i_e,\ldots,j\}$. Here we also used that, thanks to \rif{exit time}, we have
\eqn{E sum2}
$$
a_{i_e} \leq \left(\mean{Q_{i_e}} |Du|^{p-1} \, dx \, dt \right)^{1/(p-1)} \leq C_{i_e} \leq \frac{\lambda}{1000}\,.
$$

\subsection*{Step 9: Verification of $\textrm{Ind}_1(j+1)$.} Here we prove \rif{gra3},
thereby concluding the proof. We actually have to prove that
\eqn{vali}
$$
\sum_{m=-1}^{0}\left(\mean{Q_{j+m+1}} |Du|^{p-1} \, dx \, dt \right)^{1/(p-1)} \leq \lambda\,.
$$
To this end, we estimate using~\eqref{ind u-w(j-1)} (with $l=1,2$) as follows:
\begin{eqnarray}
\nonumber
&& \sum_{m=-1}^{0}\left(\mean{Q_{j+m+1}} |Du|^{p-1} \, dx \, dt \right)^{1/(p-1)}  \\ &&  \nonumber \qquad \leq   \sum_{m=-1}^{0}   \left(\mean{Q_{j+m+1}} |Du-Dw_{j-1}|^{p-1} \, dx \, dt\right)^{1/(p-1)}
\\ \nonumber && \qquad \qquad +\sum_{m=-1}^{0} \left(\mean{Q_{j+m+1}} |Dw_{j-1}|^{p-1} \, dx \, dt\right)^{1/(p-1)}
\\ \label{p-1 -> 1} && \qquad \leq   \frac{ \sigma^N}{2^N10^5} \lambda + \sum_{m=-1}^{0}\left(\mean{Q_{j+m+1}} |Dw_{j-1}|^{p-1} \, dx \, dt\right)^{1/(p-1)}\,.
 \end{eqnarray}
We further estimate the latter term on the right hand side by simply applying triangle inequality as follows:
\begin{eqnarray*}
&& \sum_{m=-1}^{0} \mean{Q_{j+m+1}} |Dw_{j-1}|^{p-1} \, dx \, dt \\ &&  \leq  \sum_{m=-1}^{0} \mean{Q_{j+m+1}} |Dw_{j-1}|^{p-2} |(Du)_{Q_{j+m+1}}| \, dx \, dt
\\ &&   \quad +\sum_{m=-1}^{0} \mean{Q_{j+m+1}} |Dw_{j-1}|^{p-2} \left(|Du-(Du)_{Q_{j+m+1}}| + |Du-Dw_{j-1}| \right) \, dx \, dt \,.
 \end{eqnarray*}
Using Young's inequality with conjugate exponents $(p-1)/(p-2)$ and $p-1$ (only when $p>2$) we get
\begin{eqnarray*}
&&
\sum_{m=-1}^{0} \mean{Q_{j+m+1}} |Dw_{j-1}|^{p-2} |(Du)_{Q_{j+m+1}}| \, dx \, dt \\
&& \qquad \qquad \leq \frac{p-2}{p-1}\sum_{m=-1}^{0} \mean{Q_{j+m+1}} |Dw_{j-1}|^{p-1} \, dx \, dt \\ &&\qquad \qquad \qquad \qquad + \frac{1}{p-1}\sum_{m=-1}^{0} |(Du)_{Q_{j+m+1}}|^{p-1}\,.
 \end{eqnarray*}
Matching the inequalities in the last two displays, reabsorbing terms, and using~\eqref{Dw(j-1) bound}, yields
\begin{eqnarray*}
&&  \sum_{m=-1}^{0} \mean{Q_{j+m+1}} |Dw_{j-1}|^{p-1} \, dx \, dt  \leq \sum_{m=-1}^{0} |(Du)_{Q_{j+m+1}}|^{p-1}
\\ & &   +
 (p-1) (A \lambda)^{p-2}  \sum_{m=-1}^{0} \mean{Q_{j+m+1}}\left(|Du-(Du)_{Q_{j+m+1}}| + |Du-Dw_{j-1}| \right) \, dx \, dt \,.
 \end{eqnarray*}
Notice that we have used that $Q_j \subset \frac12 Q_{j-1}$ in order to apply \rif{Dw(j-1) bound}. Now we estimate all the terms in the right hand side of the above inequality. Using~\eqref{a bound} and~\eqref{E sum2} (this last one only when $j=i_e$) from Step 7 we have
$$
 \sum_{m=-1}^{0} |(Du)_{Q_{j+m+1}}|^{p-1} = a_{j+1}^{p-1}+a_{j}^{p-1} \leq  2\left( \frac{\lambda}{200}\right)^{p-1}\,,
$$
while using~\eqref{E sum} gives
$$
 \sum_{m=-1}^{0} \mean{Q_{j+m+1}}|Du-(Du)_{Q_{j+m+1}}|  \, dx \, dt = E_{j+1}+E_{j} \leq \frac{ \sigma^{N}\lambda }{500}\,.
$$
Finally, using~\eqref{ind u-w(j-1)} (with $l=1,2$) and H\"older's inequality yields
$$
 \sum_{m=-1}^{0} \mean{Q_{j+m+1}} |Du-Dw_{j-1}|  \, dx \, dt \leq  \frac{ \sigma^N}{2^N 10^5}\lambda\,.
$$
Connecting the inequalities in the last four displays, and recalling the very definitions of $A$ and $\sigma$ in~\rif{A,B choices} and \eqref{sigma} respectively, gives us
$$
 \sum_{m=-1}^{0} \mean{Q_{j+m+1}}|Dw_{j-1}|^{p-1} \, dx \, dt %&\leq & 2\left( \frac{\lambda}{200}\right)^{p-1} +(p-1) (A \lambda)^{p-2} \sigma^N \lambda \\ &\leq  &
 \leq \left( \frac{\lambda}{8}\right)^{p-1}.
$$
Inserting the last inequality into~\eqref{p-1 -> 1} %and using the elementary inequalities
%$$
%\left[a^{1/(p-1)}+b^{1/(p-1)}\right]\leq 2^{(p-2)/(p-1)}(a+b)^{1/(p-1)}\leq 2(a+b)^{1/(p-1)}
%$$
%for $a,b\geq 0$,
leads to \rif{vali}. Therefore \rif{gra3} is verified and the proof of Theorem \ref{main1} is complete.
\subsection{Proof of Theorem \ref{main3}} With the standard cylinder $Q_{r}\equiv Q_r(x_0,t_0) \subset \Omega_T$ being fixed in the statement,
let us consider the function
$
h(\lambda) := \lambda - cA(\lambda)
$
where
\begin{eqnarray*}
A(\lambda)&:=&\lambda^{\frac{p-2}{p-1}}\left(\frac{1}{|Q_r|}\int_{Q_r^\lambda} (|Du|+s+1)^{p-1} \, dx \, dt\right)^{1/(p-1)}  +{\bf I}_{1, \lambda}^\mu(x_0,t_0;r)
\\&:=&\left(\mean{Q_r^\lambda} (|Du|+s+1)^{p-1} \, dx \, dt\right)^{1/(p-1)}  +{\bf I}_{1, \lambda}^\mu(x_0,t_0;r)
\end{eqnarray*}
and $c>1$ is again the constant appearing in Theorem \ref{main1}; it depends only on $n,p,\ratio, \omega(\cdot)$. Here it is $Q_r^\lambda \equiv  Q_r^\lambda(x_0,t_0)$. We are actually considering the function $h(\cdot)$ to be defined for all those positive $\lambda$ which are such that $Q_r^\lambda \subset \Omega_T$; observe that since $p \geq 2$ then the domain of definition of $h(\cdot)$ includes $[1, \infty)$ as $Q_r^\lambda\subset Q_{r}\subset \Omega_T$ when $\lambda \geq 1$.
Moreover, observe that again when $\lambda \geq 1$ we have
\eqn{hA}
$$
A(\lambda)\leq \lambda^{\frac{p-2}{p-1}}\left(\mean{Q_r} (|Du|+s+1)^{p-1} \, dx \, dt\right)^{1/(p-1)}  +{\bf I}_{1}^\mu(x_0,t_0;r)\,.
%\\&:=&\left(\mean{Q_r^\lambda} (|Du|+s+1)^{p-1} \, dx \, dt\right)^{1/(p-1)}  +{\bf I}_{1, \lambda}^\mu(x_0,t_0;r)
$$
The function $h(\cdot)$ is obviously continuous and moreover $h(1)<0$ since $c > 1$ and $A(\lambda)\geq 1$. On the other hand, \rif{hA} implies that $ h(\lambda) \to \infty$ as $\lambda \to \infty$. It follows that there exists a number $\lambda> 1$ such that $h(\lambda)=0$ and therefore $\lambda = cA(\lambda)$. In particular, $\lambda$ satisfies \rif{unoimpli}.
Therefore we can apply Theorem \ref{main1} that gives
$$
\lambda + |Du(x_0,t_0)| \leq 2\lambda = 2cA(\lambda)\,.
$$
Using in turn Young's inequality with conjugate exponents $(p-1, (p-1)/(p-2))$ when $p > 2$ and \rif{hA}, we have
$$
2cA(\lambda) \leq   \frac{\lambda}{2} + \tilde c\mean{Q_r} (|Du|+s+1)^{p-1} \, dx \, dt+2c{\bf I}_{1}^\mu(x_0,t_0;r)
$$
where $\tilde c$ depends only on $n,p,\ratio, \omega(\cdot)$. The proof follows connecting the inequalities in the last two displays.
\section{Proof of Theorems \ref{mainc1} and \ref{mainv1}}
\subsection{Proof of Theorem \ref{mainv1}} The proof consists of several steps; some of the arguments of the proof of Theorem \ref{main1} will be re-proposed in order to, this time, control the degeneracy rate of the equation.

\subsection*{Step 1: Basic setup and smallness conditions.} Since we are assuming that ${\bf I}_{1}^\mu(x_0,t_0;r)$ is locally bounded for some $r>0$ then by Theorem \ref{main3} we have that $Du$ is locally bounded in $\Omega_T$, too. Moreover, since we are proving a local statement, up to passing to open subsets compactly contained in $\Omega_T$, we can assume w.l.o.g.~that the gradient is globally bounded, therefore letting
\eqn{il_lambda}
$$
\lambda:= \|Du\|_{L^\infty(\Omega_T)}+s +1< \infty\,.
$$
From now on our analysis will proceed on cylinders of the type $Q_{r}^{\lambda} \subset  \Omega_T$. We shall prove that for every $\ep \in (0,1)$, there exists a radius $r_\ep \equiv r_\ep(n,p,\ratio, \omega(\cdot), \mu(\cdot), \eps)>0$ such that \eqn{eq: VMO goal2}
$$
E(Du, Q_\varrho^{\lambda}) =\mean{Q_\varrho^{\lambda}}|Du-(Du)_{Q_\varrho^{\lambda}}| \, dx\, dt  < \lambda\ep$$ holds whenever $\varrho \in (0,r_\ep]$ and $Q_\varrho^{\lambda} \subset  \Omega_T$. Once this fact is proved the VMO-regularity of $Du$ in the sense of Theorem \ref{mainv1} follows by an easy change-of-variables argument as $\lambda$ is now fixed in \rif{il_lambda}.

Towards the application of Theorems \ref{thm:Dw(j) continuous}-\ref{sublime}, and with $c_2\equiv c_2(n,p,\nu,L)$ and $c_3\equiv c_3(n,p,\nu,L)$ being as in Corollary~\ref{cor:Dw(j) bounded 2} and Theorem~\ref{thm:Dv(j) bounded}, respectively, we start fixing \begin{equation}\label{A,B choices-c}
    A := \frac{1000^{4pN} \max\{c_2,c_3,200\}}{\eps}\,, \quad B := \frac{10^5}{\eps}\,, \quad \delta := \frac{\eps}{10^5}\,, \quad \bar \eps := \frac{\eps}{10^{5N}} \,.
\end{equation}
With the choice in \rif{A,B choices-c} we determine $\sigma_1 \equiv \sigma_1(n,p,\nu,L, \omega(\cdot), \eps) $ and $\sigma_2\equiv \sigma_2(n,p,\linebreak \nu,L, \eps)$ from Theorems~\ref{thm:Dw(j) continuous} and \ref{sublime}, respectively. Next, we fix \begin{equation}\label{sigmac}
  \sigma := \min\{\sigma_1,\sigma_2, 16^{-N}\} \in (0,1/4)
\end{equation}
and notice that $\sigma$ only depends on $n,p,\ratio, \omega(\cdot),\ep$. Such choices, looking at Lemmas \ref{lemma:u-w comparison}, \ref{lemma:w-v} and \ref{Du vs Dv(j)}, determine the constants $\bar c_1,  \bar c_5,\bar c_6$ and $\bar c_7$ as depending only on $n,p,\ratio, \omega(\cdot),\ep$. We again need some limitations on the size of the radii considered; here $R_1\equiv R_1(n,p,\ratio, \omega(\cdot))$ still denotes the radius considered in Theorem \ref{thm:Dw(j) bounded 1} and Corollary \ref{cor:Dw(j) bounded 2}. Moreover, we select a new radius $R_3\equiv R_3(n,p,\ratio, \omega(\cdot),  \mu(\cdot), \eps)$ in such a way that the following {\em smallness conditions} hold:
\eqn{pigpic}
$$
\sup_{0<\varrho\leq R_3}\,
 \sup_{(x, t)  \in \Omega_T} %\frac{|\mu|(Q_{\varrho}^\lambda(x, t))}{ \lambda\varrho^{N-1}}  +
 \left[\frac{|\mu|(Q_{\varrho}^\lambda(x,t))}{ \lambda\varrho^{N-1}}  \right]^{(n+2)/[(p-1)n+p]}\leq \frac{\sigma^{3N} \ep}{10^6 \bar c_1  \bar c_5\bar c_6 \bar c_7}
 $$
 and
 \eqn{pigpicw}
$$ [\omega(R_3)]^{2/p} \leq \frac{\sigma^{3N} \ep}{10^6 \bar c_1  \bar c_5\bar c_6 \bar c_7}$$
%\end{eqnarray}
 and we this time set $\bar R_0:= \min\{R_1, R_3\}/4$; as a consequence we have $\bar R_0\equiv \bar R_0 (n,p,\ratio, \omega(\cdot), \mu(\cdot), \eps)$. Notice only that the possibility of gaining \rif{pigpic} follows directly by the assumption \rif{assvmo}, and from the fact that, since $\lambda \geq 1$ and $p \geq 2$, then $Q_{\varrho}^\lambda(x,t) \subset Q_{\varrho}(x,t)$.
We now fix a cylinder $Q \equiv Q_{r}^{\lambda}(x_0,t_0) \subset \Omega_T$ with $r \in (\sigma \bar R_0,\bar R_0]$, and accordingly to the setup defined in Section \ref{comp maps} define the chain of shrinking parabolic intrinsic cylinders as follows:
\eqn{sceltaitera}
$$
Q_j \equiv Q_{r_j}^{\lambda}(x_0,t_0)\,, \qquad r_j = \sigma^j r, \qquad \mbox{where}\ \ r \in (\sigma \bar R_0,\bar R_0]\,,
$$
for every integer $j \geq 0$. The related comparison solutions $w_j$ and $v_j$ are accordingly defined as in \rif{CD-local} and \rif{CD-local v}, respectively; finally, we denote
$$
E_j := \mean{Q_{j}} |Du-(Du)_{Q_j}| \, dx \, dt\,.
$$
We shall preliminary prove that
\eqn{eq:excess conclusion 1}
$$
E_{j+1} < \lambda\ep \qquad \qquad \forall \, j \in \mathbb{N} \cap[1,\infty)\,.
$$
Now, it is obvious that if for a given $j \geq 1$ it occurs that
\eqn{primo}
$$
\left(\mean{Q_{j+1}} |Du|^{p-1} \, dx \, dt \right)^{1/(p-1)} < \frac{\lambda\ep}{50} \, ,
$$
then \rif{eq:excess conclusion 1} holds, therefore we can confine ourselves to assume that \rif{primo} does not hold. To prove \rif{eq:excess conclusion 1} in this last case we start proving the following implication, which is valid whenever $j \geq 1$:
\begin{eqnarray}
&& \nonumber \left( \mean{Q_{j+1}} |Du|^{p-1} \, dx \, dt \right)^{1/(p-1)}  \geq \frac{\lambda\ep}{50} \\ %%\eqn{eq:excess decay 1000}
&& \qquad \qquad \Longrightarrow  E_{j+1} \leq \frac{\eps}{300} E_j + \frac{2\bar c_6}{\sigma^{N}}  \omega(r_j) \lambda + \frac{2\bar c_7}{\sigma^{N}} \left[\frac{|\mu|(Q_{j-1})}{r_{j-1}^{N-1}} \right]\,. \label{decaydecay}
\end{eqnarray}
%%%TUOMO : check one more time below %%%CHECK %%%ROS simplified language below
Indeed, as in any case we have $E_{j}\leq 2 \lambda$ and \rif{pigpic}-\rif{pigpicw} hold, it follows $E_{j+1}\leq \eps\lambda/50$ as a direct consequence of \rif{decaydecay} when \rif{primo} does not hold. Notice that we have used that $(n+2)/[(p-1)n+p]\leq 1$, which holds since we are considering the case $p\geq 2$. Therefore to prove \rif{eq:excess conclusion 1} we are reduced to check the validity of \rif{decaydecay}.

\subsection*{Step 2: Proof of \rif{decaydecay}.} We note that the setting of the shrinking cylinders $\{Q_j\}$ and related comparison maps $v_j, w_j$ defined in \rif{sceltaitera} is completely similar to the one adopted in the proof of Theorem \ref{main1}, only the choice of the constants differs. To prove \rif{decaydecay} is in turn sufficient to prove the following group of inequalities for the choices made in \rif{A,B choices-c}:
\eqn{ww2}
$$
\sup_{\frac 12Q_{j}} |Dw_{j}|+s \leq A\lambda
$$
\eqn{ww1}
$$
\frac{\lambda}{A}\leq  |Dw_{j-1}| \leq |Dw_{j-1}| + s \leq A	\lambda\qquad \mbox{in} \ Q_{j}
$$
\eqn{ww3}
$$
 \frac{\lambda}{B} \leq  \sup_{ Q_{j+1}}\, |Dv_{j}| \leq  \sup_{\frac 14 Q_{j}}\, |Dv_{j}| + s \leq A	\lambda\,.
$$
Indeed, taking these for granted, let us see how to conclude with the proof of \rif{decaydecay}.
Inequalities in displays \rif{ww2}-\rif{ww3} are completely analogous to \rif{ind Dw(j) 1} and \rif{Dw(j-1) upper and lower}-\rif{Dv(j) upper and lower}, respectively. We can therefore exactly argue as in Step 6 of the proof of Theorem \ref{main1}: using \rif{ww2}-\rif{ww1} we get \rif{decisive comp} by Lemma \ref{Du vs Dv(j)}, while \rif{ww3} allows to use Theorem~\ref{sublime} with $\bar \eps=10^{-5N} \eps $ thereby yielding
$$
\mean{Q_{j+1}} |Dv_j-(Dv_{j})_{Q_{j+1}}| \, dx \, dt \leq \frac{\eps}{10^{5N}} \mean{\frac14 Q_{j}} |Dv_j-(Dv_{j})_{\frac14 Q_{j}}| \, dx \, dt
$$
which plays the role of
\rif{decisive comp2} in this context. Proceeding as in Step 6 we finally arrive at the inequality in display \rif{decaydecay}, which is the analog of \rif{summa} in this context.

It remains to prove \rif{ww2}-\rif{ww3}. Using Lemma \ref{lemma:u-w comparison} and \rif{pigpic}, we have
\begin{eqnarray}
&&\nonumber  \ \ \left(\mean{Q_{j+1}} |Du-Dw_{j-1}|^{p-1} \, dx \,dt\right)^{1/(p-1)}   \\ & & \qquad
 \leq \sigma^{-\frac{2N}{p-1}}\left(\mean{Q_{j-1}} |Du-Dw_{j-1}|^{p-1} \, dx \,dt\right)^{1/(p-1)}\leq  \frac{\lambda\ep}{10^4}\,.
\label{keep1}\end{eqnarray}
By \rif{keep1}, keeping \rif{il_lambda} in mind, and using triangle inequality we get
$$\left(\mean{Q_{j-1}} |Dw_{j-1}|^{p-1} \, dx \,dt\right)^{1/(p-1)}\leq 2\lambda$$
so that, applying Corollary~\ref{cor:Dw(j) bounded 2} to $w_{j-1}$ we finally come to
\eqn{ww0}
$$
\sup_{Q_{j}} |Dw_{j-1}| + s \leq \sup_{\frac 12 Q_{j-1}} |Dw_{j-1}| + s \leq A	\lambda
$$
that proves the right hand side inequality in \rif{ww1}. In a completely similar way we obtain also
\rif{ww2}. We now prove the left hand side inequality in \rif{ww1}.
Again using \rif{keep1}, we get
\begin{eqnarray}
\nonumber
\sup_{Q_{j}} |Dw_{j-1}| &\geq & \left(\mean{Q_{j+1}} |Du|^{p-1} \, dx \,dt\right)^{1/(p-1)}\\ &&   \hspace{-3mm}-\left(\mean{Q_{j+1}} |Du-Dw_{j-1}|^{p-1} \, dx \,dt\right)^{1/(p-1)}\geq  \frac{ \lambda\ep}{50 } - \frac{\lambda\ep}{10^4}\geq  \frac{ \lambda\eps }{100}. \nonumber%\label{passa0}
\end{eqnarray}
The lower bound in the last display provides the existence of a point $(\tilde x, \tilde t) \in Q_{j}$ such that $|Dw_{j-1}(\tilde x, \tilde t)|\geq \lambda\eps  /200$ while
\rif{ww0} allows to apply Theorem \ref{thm:Dw(j) continuous} (to $w_{j-1}$ in $\frac 12 Q_{j-1}$) getting, thanks to the choice of $\delta$ in \rif{A,B choices-c},
$$
\osc_{Q_{j}}\, Dw_{j-1}\leq \frac{\lambda\eps }{10^5}\,.%\lambda \eps /10^5\,.
$$
The last two inequalities finally give $|Dw_{j-1}| \geq \lambda \eps /10^5$ in $Q_{j}$. Summarizing, recalling the choice of $A$ in \rif{A,B choices-c} and again \rif{ww0}, the proof \rif{ww1} turns out to be complete. It remains to prove \rif{ww3}. Using~\eqref{ww2}, Lemma~\ref{lemma:w-v} and finally \rif{pigpicw} we have
$$\left(\mean{\frac12 Q_{j}} |Dw_j-Dv_j|^{p-1} \, dx \,dt \right)^{1/(p-1)}
\leq  \bar c_5^{1/p}\left[\omega(r_{j})\right]^{2/p} \lambda  \leq \frac{\sigma^{3N}\lambda \ep}{10^6}\,.$$
Combining the above estimate with~\eqref{ce} and \rif{pigpic} gives, after a few standard manipulations,
\begin{eqnarray}
&& \nonumber \max\left\{\left(\mean{Q_{j+1}} |Du-Dv_{j}|^{p-1} \, dx \,dt \right)^{\frac{1}{p-1}},\left(\mean{\frac 12Q_{j}} |Du-Dv_{j}|^{p-1} \, dx \,dt \right)^{\frac{1}{p-1}}\right\} \\ && \qquad \qquad \leq \frac{\bar c_1}{\sigma^N} \lambda \left[\frac1\lambda \frac{|\mu|(Q_j)}{r_j^{N-1}} \right]^{(n+2)/[(p-1)n+p]}+\frac{\sigma^{2N} \lambda\ep}{10^6}\leq \frac{\sigma^{2N} \lambda\ep}{10^5}\,.\label{passa}
\end{eqnarray}
Therefore, using \rif{passa}, Minkowski's inequality and recalling the definition of $\lambda$ in \rif{il_lambda} we have
$$
\left(\mean{\frac12 Q_{j}} |Dv_{j}|^{p-1} \, dx \,dt\right)^{1/(p-1)}\leq \lambda +\frac{\sigma^{2N} \lambda\ep}{10^5}\leq 2 \lambda
$$
so that the right hand side inequality in \rif{ww3} follows by Theorem \ref{thm:Dv(j) bounded}. As for the left hand side, notice that
the first inequality in \rif{decaydecay} and \rif{passa} imply
\begin{eqnarray*}
\nonumber
 &&\sup_{Q_{j+1}} |Dv_{j}| \geq  \left(\mean{Q_{j+1}} |Du|^{p-1} \, dx \,dt\right)^{1/(p-1)}\\ &&   \qquad-\left(\mean{Q_{j+1}} |Du-Dv_{j}|^{p-1} \, dx \,dt\right)^{1/(p-1)}  \geq \frac{\lambda\ep}{50} - \frac{\sigma^{2N} \lambda\ep}{10^5}  \geq \frac{\eps\lambda}{10^5}\equiv \frac{\lambda}{B}
\end{eqnarray*}
so that \rif{ww3} is completely proved.

\subsection*{Step 3: Interpolation of radii.} With \rif{eq:excess conclusion 1} at our disposal we can finally conclude the proof of Theorem \ref{mainv1} by letting $
r_\ep := \sigma^2 \bar R_0.$ %Since all the estimates above are uniform with respect to the initial cylinder considered $Q_{r}^{\lambda}$ and especially with respect to
%{\em the initial radius} $r \in (\sigma ,\bar R_0]$ chosen to build the chain in \rif{sceltaitera}, the inequality in
%\rif{eq: VMO goal2} follows.
Indeed, consider $\varrho \leq \sigma^2\bar R_0$; this means there exists an integer $m\geq 2$ such that $\sigma^{m+1} \bar R_0< \varrho \leq \sigma^{m} \bar R_0$. Therefore we have $\varrho = \sigma^{m}r$ for some $r \in (\sigma \bar R_0, \bar R_0]$ and \rif{eq: VMO goal2} follows from \rif{eq:excess conclusion 1} with this particular choice of $r$.

\subsection{Proof of Theorem \ref{mainc1}}
The proof is now based on a combination of the arguments of Theorem \ref{mainv1} with those which are more typical of the elliptic case; we report everything in full detail for the sake of completeness and readability, and also because a certain number of modifications is really needed. We shall therefore keep the notation introduced in Step 1 of the proof of Theorem \ref{mainv1}. Essentially, we are going to use the same choices in \rif{A,B choices-c}-\rif{pigpicw} but using an additional smallness condition on the radii used; in this way we can use both the inequalities in the proof of Theorem \ref{mainv1} and the result of Theorem \ref{mainv1}. We consider a cylinder $Q_0 \Subset \Omega_T$ and prove that for every $\ep>0$ there exists a radius $r_\ep\leq d_{\rm par} (Q_0, \partial \Omega_T)/2$, depending only on $n,p,\ratio, \omega(\cdot), \mu(\cdot), \eps$, such that
\eqn{smally2}
$$|(Du)_{Q_{\varrho}^{\lambda}(x_0,t_0)}-(Du)_{Q_{\rho}^{\lambda}(x_0,t_0)}|\leq  \lambda\ep\qquad \mbox{holds for every}  \ \varrho, \rho \in (0,r_\ep]$$
whenever $(x_0,t_0) \in Q_0$. This proves that $Du$ is the local uniform limit of continuous maps - defined via the averages - and hence it is continuous. The rest of the proof goes in two steps.

\subsection*{Step 1: Dyadic sequences and continuity.} To begin with the proof of \rif{smally2} we recall that $R_1$ is determined in Theorem \ref{thm:Dw(j) bounded 1}. Moreover $R_3$ is determined in \rif{pigpic}-\rif{pigpicw} with the constant $\bar c_1,  \bar c_5,\bar c_6$ and $\bar c_7$ obtained Theorem \ref{mainv1} and corresponding to the choices made in \rif{A,B choices-c} and \rif{sigmac}. Next, we take yet another positive radius $R_4\leq d_{\rm par} (Q_0, \partial \Omega_T)/2$ such that
\eqn{unap0}
$$
\sup_{(x, t) \in Q_0} \, \int_0^{4 R_4} \frac{|\mu|(Q_{\varrho}(x, t))}{\lambda \varrho^{N-1}} \, \frac{d\varrho}{\varrho} +\int_{0}^{4R_4} \omega(\varrho) \, \frac{d\varrho}{\varrho} \leq\frac{\sigma^{4N} \ep}{10^6 \bar c_6 \bar c_7}%\frac{\sigma^{4N} \ep}{10^6 \bar c_1  \bar c_5\bar c_6 \bar c_7}
$$
and, recalling the definition in \rif{excess}
\eqn{duep}
$$
 \sup_{0<\varrho\leq R_4} \, \sup_{(x, t) \in Q_0} \, E(Du, Q_{\varrho}^{\lambda}(x, t)) \leq  \frac{ \sigma^{4N}\lambda\ep }{10^5}\,.
$$
Observe that it is possible to make the choice in \rif{duep} thanks to Theorem \ref{mainv1} and in particular to \rif{eq: VMO goal2}, that ensures that $R_4$ can be chosen in a way that makes it depending only on $n,p,\ratio, \omega(\cdot), \mu(\cdot), \eps$. Finally, this time we set $\tilde R_0:= \min\{R_1, R_3, R_4\}/2$ and take everywhere $r \leq \tilde R_0$ so that $\tilde R_0$ ultimately depends again on $n,p,\ratio, \omega(\cdot), \mu(\cdot), \eps$ only. The sequence of shrinking cylinders $\{Q_j\}$ is now defined as
$$
Q_j \equiv Q_{r_j}^{\lambda}(x_0,t_0)\,, \qquad r_j = \sigma^j \tilde R_0, \quad \mbox{for}\ \ j \geq 0\,,
$$
while $\lambda \geq 1$ is still defined as in \rif{il_lambda} and $\sigma$ in \rif{sigmac}. By \rif{unap0}, computations similar to those in \rif{compute} and \rif{compute0} then give
\eqn{unap}
$$
 \sum_{i=0}^\infty \frac{|\mu|(Q_i)}{ \lambda r_i^{N-1}  } +  \sum_{i=0}^\infty \omega(r_i)\leq\frac{\sigma^{3N} \ep}{10^6 \bar c_6\bar c_7}\,.% \frac{\sigma^{3N} \ep}{10^6 \bar c_1\bar c_5\bar c_6\bar c_7}\,.
$$
In Step 2 we will prove that
\eqn{sm1}
$$|(Du)_{Q_h}-(Du)_{Q_k}|\leq \frac{\lambda\ep}{12}\qquad \mbox{holds whenever} \ 2 \leq k \leq  h\,.$$
Here we show how to use \rif{sm1} to finish the proof and to verify \rif{smally2} with the choice $r_\ep:=\sigma^2\tilde R_0$. Indeed, let us fix $0 < \rho < \varrho \leq r_\ep$. This means that there exist two integers, $2 \leq k \leq h$, such that
$
\sigma^{k+1}\tilde R_0 < \varrho \leq  \sigma^{k}\tilde R_0$ and $\sigma^{h+1}\tilde R_0 < \rho \leq  \sigma^{h}\tilde R_0$.
Applying~\rif{duep} we get
\begin{eqnarray*}
\nonumber |(Du)_{Q_{\varrho}^{\lambda}(x_0,t_0)}- (Du)_{Q_{k+1}}|& \leq &
\mean{Q_{k+1}} |Du- (Du)_{Q_{\varrho}^{\lambda}(x_0,t_0)}|\, dx \, dt\\
\nonumber & \leq &
\frac{|Q_{\varrho}^{\lambda}(x_0, t_0)|}{|Q_{k+1}|}\mean{Q_{\varrho}^{\lambda}(x_0, t_0)} |Du- (Du)_{Q_{\varrho}^{\lambda}(x_0, t_0)}|\, dx \, dt \\ & \leq &
\sigma^{-N}E(Du,Q_{\varrho}^{\lambda}(x_0, t_0) ) \leq \frac{\lambda\ep}{10}\,,
\end{eqnarray*}
and, similarly,
$$|(Du)_{Q_{\rho}^{\lambda}(x_0, t_0)}- (Du)_{Q_{h+1}}|\leq \frac{\lambda\ep}{10}\;.$$ Using the inequalities in the last two displays together with \rif{sm1} and triangle inequality establishes~\rif{smally2} and the proof is complete, modulo the content of the next and final step.

\subsection*{Step 2: Proof of \rif{sm1}.} The preliminary observation to make is that, with the choices made here, \rif{decaydecay} holds in this setting as this ultimately relies on \rif{pigpic}-\rif{pigpicw}, that indeed are in force by the choice of $\tilde R_0$. To continue, let us consider the set $\mathcal{L}$ defined by
$$
\mathcal{L} := \left\{ j \in \mathbb{N} \ : \ \left(\mean{Q_{j}} |Du|^{p-1} \, dx \, dt\right)^{1/(p-1)}  < \frac{\lambda\ep}{50} \right\} \, ,
$$
and, accordingly, we then define the sets
$$
\mathcal C_i^ m= \{ j \in \en \, : i \leq j \leq  i+m,  \ i \in  \mathcal{L},\  j \not\in \mathcal{L} \ \mbox{if}\ j >i\}
$$ for $m \in \en$
%$$
%\mathcal C_i^\infty= \{ j \in \en \, : i \leq j <  \infty,  \ i \in \mathcal{L},  \ j \not\in \mathcal{L} \ \mbox{if}\ j >i\}
%$$
and, finally, the number
$j_e := \min\, \mathcal{L}.$ Note that it may happen that $j_e=	\infty$; in this case $\mathcal{L}$ is empty and the first inequality in \rif{decaydecay} holds for every $j \geq 1$. The idea is now to employ \rif{decaydecay} on suitable sets $\mathcal C_i^ m$ using the indexes $i$ as a sort of exit time indexes; the difference is that they can be countably many now. We can now prove \rif{sm1}, obviously assuming $k < h$. The first case we analyze is when $ k < h \leq j_e$; %If $k+1 \leq h \leq j_e$
we use \rif{decaydecay} and the definition of $j_e$ to infer that the inequality
\eqn{doneafter}
$$
E_{j+1} \leq \frac{1}{2} E_j + \frac{2 \bar c_6}{\sigma^{N}} \omega(r_j) \lambda + \frac{2 \bar c_7}{\sigma^{N}} \left[\frac{|\mu|(Q_{j-1})}{r_{j-1}^{N-1}} \right]
$$
holds for every $j \in \{k-1, \ldots, h-2\}$. Summing up the previous inequalities easily yields
$$
\sum_{i=k}^{h-1} E_i \leq E_{k-1} + \frac{4 \bar c_6}{\sigma^{N}}\sum_{j=0}^\infty \omega(r_j)\, \lambda + \frac{4 \bar c_7}{\sigma^{N}} \sum_{j=0}^\infty \frac{|\mu|(Q_{j})}{r_{j}^{N-1}} \leq \frac{\sigma^{2N}\lambda\ep}{50}
$$
where we have used \rif{duep}-\rif{unap}, therefore \rif{sm1} follows since % When $h=k+1$, the previous inequality is a direct consequence of \rif{duep}. Summarizing, when $k < h \leq j_e$, \rif{sm1} follows since
\begin{eqnarray}
|(Du)_{Q_h}-(Du)_{Q_{k}}| & \leq & \sum_{i=k}^{h-1}  |(Du)_{Q_{i+1}}-(Du)_{Q_i}|
\nonumber \\ & \nonumber \leq &
 \sum_{i=k}^{h-1}  \mean{Q_{i+1}}|Du-(Du)_{Q_i}| \, dx \, dt
\\ & \leq & \nonumber
 \sum_{i=k}^{h-1} \frac{|Q_{i}|}{|Q_{i+1}|} \mean{Q_{i}}|Du-(Du)_{Q_i}| \, dx \, dt \\ &=& \sigma^{-N}
 \sum_{i=k}^{h-1} E_i \leq \frac{\lambda\ep}{50}\,.  \label{eq:estimte sum}
\end{eqnarray}
The second case we consider is when $j_e\leq k < h$, where we prove \rif{sm1} through the inequalities
\eqn{sottodd}
$$
|(Du)_{Q_h}| \leq \frac{\lambda\ep}{25}\qquad \mbox{and}\qquad |(Du)_{Q_k}|\leq \frac{\lambda\ep}{25}\,.
$$
In \rif{sottodd}, we prove the former, the argument for the latter being the same when $k> j_e$, otherwise $|(Du)_{Q_k}|\leq \lambda\ep/25$ is trivial if $k=j_e\in \mathcal{L}$. If $h \in \mathcal{L}$, the first inequality in \rif{sottodd} follows immediately from the definition of $\mathcal{L}$. On the other hand, if $h \not \in \mathcal{L}$, then, as $h > j_e$, it is possible to consider a set $\mathcal C_{i_h}^{m_h}$ with $m_h>0$, such that $h \in \mathcal C_{i_h}^{m_h}$; notice that $h > i_h$ as $h \not \in \mathcal{L} \ni i_h$. Then \rif{decaydecay} gives that \rif{doneafter} holds whenever $j \in \{i_h, \ldots, i_h+m_h-1\}$. Summing up and performing elementary manipulations gives
$$
\sum_{i=i_h}^{i_h+m_h} E_i \leq 2E_{i_h} + \frac{4 \bar c_6}{\sigma^{N}} \sum_{j=0}^\infty \omega(r_j)\, \lambda+ \frac{4 \bar c_7}{\sigma^{N}} \sum_{j=0}^\infty \frac{|\mu|(Q_{j})}{r_{j}^{N-1}} \leq \frac{\sigma^{2N}\lambda\ep}{50}
$$
where again we have used \rif{duep}-\rif{unap}. Therefore, as in \rif{eq:estimte sum}, we have
$$
|(Du)_{Q_h}-(Du)_{Q_{i_h}}|\leq \sigma^{-N}
 \sum_{i=i_h}^{h-1} E_i \leq \sigma^{-N}
 \sum_{i=i_h}^{i_h+m_h} E_i \leq\frac{\lambda\ep}{50}
$$
and then, using that $|(Du)_{Q_{i_h}}|\leq \lambda\ep/50$ as $i_h \in \mathcal{L}$, we have
$$
|(Du)_{Q_{h}}| \leq |(Du)_{Q_{i_h}}|+ |(Du)_{Q_h}-(Du)_{Q_{i_h}}| \leq \frac{\lambda\ep}{25}
$$
that is \rif{sottodd}. The last case to consider is when $k < j_e < h$, that can be actually treated by a combination of the first two. It suffices to prove that the inequalities in display \rif{sottodd} still hold. Indeed, the first inequality in \rif{sottodd} follows exactly as in the second case. As for the second estimate in \rif{sottodd}, let us remark that, as $j_e \in \mathcal{L}$, we have that
$
|(Du)_{Q_{j_e}}| \leq \lambda\ep/50.
$
On the other hand, we can use the first case $k < h \leq j_e$ with $h=j_e$, thereby obtaining $|(Du)_{Q_{j_e}}-(Du)_{Q_k}|  \leq \lambda\ep/50$ and therefore the second inequality in \rif{sottodd} follows via triangle inequality.

\section{Proof of Theorem \ref{solath}}
The proof of Theorem \ref{solath} is a consequence of a few simple observations once Lemma \ref{lemma:u-w comparisonApp} is at disposal  and the sequence of comparison solutions $w_j$ is introduced. Let's start by Theorem \ref{main1}. Going back to Section \ref{comp maps}, after having introduced the maps $\{w_j\}$ in Lemma \ref{lemma:u-w comparisonApp} we can introduce the maps $\{v_j\}$ exactly as in \rif{CD-local v}. It is now easy to see that with this definition all the properties of the maps $\{w_j\}$ and $\{v_j\}$ described in Sections \ref{comp maps}-\ref{compi2}  and used in the proof of Theorem \ref{main1} hold, and especially Lemma \ref{lemma:ce3}-\ref{Du vs Dv(j)}. The only difference, which stems from the right hand sides in the inequalities in Lemma \ref{lemma:u-w comparisonApp}, is that the quantities $|\mu|(\lfloor Q_j\rfloor_{\rm par})$ appear instead of $|\mu|(Q_j)$; this is anyway irrelevant in the context of Theorem \ref{main1}, in view of \rif{countably} and of the first equality in \rif{compute0}. The proof now follows exactly as in the finite energy case. As a consequence, all the corollaries of Theorem \ref{main1}, starting by Theorem \ref{main3}, hold for SOLA as well. Next, when passing to Theorems \ref{mainc1} and \ref{mainv1} the proofs remain completely the same upon using Lemma \ref{lemma:u-w comparisonApp}; note that this time the appearance of $|\mu|(\lfloor Q_j\rfloor_{\rm par})$ instead of $|\mu|(Q_j)$ gives no problem at all since all the proofs are based on the use of smallness conditions as \rif{pigpic} and \rif{unap0}.

\begin{remark}\label{solaremark} Definition \ref{soladef} of SOLA is quite natural, and is motivated by the standard way of approximating measures with bounded functions in the weak-* convergence, via parabolic smoothing, that is, using mollifiers that, although acting uniformly in space, act backward in time. The main difference with the standard elliptic case is that, instead of getting that
$$\limsup_{h} \,|\mu_h|(Q)\leq |\mu|(\overline{Q})\,,$$
that is an inequality involving the full closure of $Q$, we get \rif{convergencemeasures} so that the upper part of a cylinder does not play any role in the approximation. This is the advantage that allows to pass to the limit easily in the pointwse potential estimates. Moreover, this in accordance to the fact that, when dealing with evolutionary equations, the behavior at a certain instant of the solution only depends on what happened at past times, but not on what happens at that instant. Let us briefly recall the procedure. One fixes a family of smooth mollifiers (approximation of the identity) $\{\phi_h\}$ with $\phi_h:= h^n\phi(x/h)$ with $\phi \in C^{\infty}_0(B_{1})$, $\phi \geq 0$ and $\|\phi\|_{L^1}=1$. Similarly, we consider another family of smooth mollifiers $\{\tilde \phi_h\}$, this time in one variable: $\tilde\phi_h:= h\tilde\phi(x/h)$ with $\tilde\phi \in C^{\infty}_0((-1, 1))$ and $\|\tilde\phi\|_{L^1}=1$.
We then define $$\mu_h:= \left[\phi_h(\cdot)\tilde \phi_h(\cdot+1/h)\right]*\mu \in L^\infty\,,$$ and notice that the mollification in the time variable is backward.  In such a way we obtain a weakly* convergent sequence (in the sense of measures) and also \rif{convergencemeasures}. Such a sequence can be for instance used in \cite{B} to derive the corresponding existence theorems. The point that we want to stress here is that, in order to give a more suitable definition of SOLA, it is very often not sufficient to take {\em any} approximation of the measure $\mu$ via weak* convergence (something that is for instance sufficient when deriving global estimates). Instead, a more careful way of approximating measures, tailored to both the geometry of the problem in question and the degree of fine properties of solutions one wants to derive, must be adopted.
\end{remark}

\subsection*{Acknowledgements.} The authors are supported by the ERC grant 207573 ``Vectorial Problems" and by the
Academy of Finland project ``Potential estimates and applications for nonlinear parabolic partial
differential equations". The authors also thank Paolo Baroni for remarks on a preliminary version of the paper.

\end{document}